\documentclass[11pt]{amsart}
\usepackage{amssymb,amsmath,txfonts}
\usepackage{hyperref}
\newtheorem{theorem}{Theorem}
\newtheorem{prop}{Proposition}
\newtheorem*{lemma}{Lemma}
\newtheorem{remark}{Remark}
\newtheorem{claim}{Claim}

\newtheorem*{definition}{Definition}
\newtheorem{cor}{Corollary}

\numberwithin{equation}{section}
\def\Xint#1{\mathchoice
  {\XXint\displaystyle\textstyle{#1}}%
  {\XXint\textstyle\scriptstyle{#1}}%
  {\XXint\scriptstyle\scriptscriptstyle{#1}}%
  {\XXint\scriptscriptstyle\scriptscriptstyle{#1}}%
  \!\int}
\def\XXint#1#2#3{{\setbox0=\hbox{$#1{#2#3}{\int}$}
  \vcenter{\hbox{$#2#3$}}\kern-.5\wd0}}

\def\dashint{\Xint-}

\author{Gang Liu}
\address{Department of Mathematics\\University of Minnesota\\Minneapolis, MN 55455}
\email{liuxx895@math.umn.edu}
\title[K\"ahler manifolds with Ricci curvature lower bound]{K\"ahler manifolds with Ricci curvature lower bound}
\date{}
\begin{document}
\maketitle
\begin{abstract}On K\"ahler manifolds with Ricci curvature bounded from below, we establish some theorems which are counterparts of some classical theorems in Riemannian geometry, for example, Bishop-Gromov's relative volume comparison, Bonnet-Meyers theorem, and Yau's gradient estimate for positive harmonic functions. The tool is a Bochner type formula reflecting the K\"ahler structure.

\end{abstract}

\bigskip
\bigskip

\section{\bf{Introduction}}

In this paper we study some geometric quantities on K\"ahler manifolds when the Ricci curvature has a lower bound. Our point of view is from Riemannian geometry. To distinguish from the Riemannian case, we derive a Bochner type formula reflecting the K\"ahler structure. One of the main results is the following:

\begin{theorem}\label{thm1}
{Let $M^m$($m >1$) be a complete K\"ahler manifold with $Ric \geq (2m-1)k$($k \neq 0$) and denote $B_x(r)$ to be the geodesic ball in $M$ centered at $x$ with radius $r$. Let $N$ be the $2m$ dimensional simply connected real space form with sectional curvature $k$ and denote $B_N(r)$ to be the geodesic ball in $N$ with radius $r$.  For any point $p \in M$ and constants $0<c<a<b$, there exists a constant $\epsilon = \epsilon(b, a, m, k)>0$ so that the area of the geodesic spheres satisfies $$\frac{A(\partial B_p(b))}{A(\partial B_p(a))} \leq \frac{A(\partial B_N(b)) }{A(\partial B_N(a)) }(1 - \epsilon).$$ Furthermore, if $k = -1$, then $\epsilon$ depends only on $c, b-a, m$.}
\end{theorem}
\begin{remark}{When the bisectional curvature is bounded from below, P. Li and J. Wang \cite{[LW]} proved the sharp version of theorem \ref{thm1} comparing with the complex space forms . However, if we only assume the Ricci curvature has a positive lower bound, one cannot expect a sharp estimate of theorem \ref{thm1} comparing with the complex space forms. The example will be given in section 5. }
\end{remark}

 Theorem \ref{thm1} has several corollaries:
\begin{cor}\label{cor1}
{ Using the same notation as in theorem \ref{thm1}, we have $$\frac{Vol(B_p(b))}{Vol(B_p(a))} \leq \frac{Vol(B_N(b)) }{Vol(B_N(a))}(1 - \epsilon)$$ where $\epsilon = \epsilon(b, a, m, k)>0$. If $k = -1$, $\epsilon$ depends only on $b-a , c, m$.}
\end{cor}

\begin{definition}{ Let $(M^n, g)$ be a complete Riemannian manifold. Choose a point $p \in M$, define the volume entropy of $M$ to be $h(M, g) = \overline{\lim\limits_{r \to + \infty}} \frac{\ln Vol(B_p(r))}{r}$ where $B_p(r)$ is the geodesic ball in $M$ centered at $p$ with radius $r$. }
\end{definition}
\begin{cor}\label{cor2}
{ Let $M^m$($m >1$) be a complete K\"ahler manifold with $Ric \geq -(2m-1)$,  then the volume entropy $h(M)$ satisfies $$h(M) \leq 2m-1-\epsilon$$ where $\epsilon$ is a positive constant depending only on $m$.}
\end{cor}
\begin{cor}\label{cor3}
{ Let $M^m$($m >1$) be a complete K\"ahler manifold with $Ric \geq (2m-1)$,  then the diameter $d(M)$ satisfies $$d(M) \leq \pi-\epsilon$$ where $\epsilon$ is a positive constant  depending only on $m$.}
\end{cor}

\begin{cor}\label{cor4}
{ Under the same assumption as in theorem \ref{thm1}, let $\lambda_1$ be the first eigenvalue of the Laplacian with Dirichlet boundary condition, then we have $$\lambda_1(B_p(r)) \leq \lambda_1( B_{N}( r)) - \epsilon$$ where $\epsilon$ is a positive constant depending only on $m$, $k$ and $r$.}
\end{cor}

\begin{remark}{The corollaries above are counterparts of Bishop-Gromov volume comparison theorem \cite{[BC]}, Bonnet-Meyers theorem \cite{[CE]}, Cheng's spectrum estimate \cite{[C]}.}
\end{remark}

 Given a stronger condition in theorem \ref{thm1},  we can obtain a better result. Explicitly, we have the following:
\begin{theorem}\label{thm2}
{Let $M^m$($m >1$) be a complete K\"ahler manifold with $Ric \geq (2m-1)k$, $k \neq 0$. Let $N$ be the $2m$ dimensional simply connected real space form with sectional curvature $k$
. For a point $p\in M$, denote $r_M(x)$ to be distance function from $p$ to $x$ in $M$. Let $r_N$ be the distance function on $N$. If $r \leq \frac{i_0}{2}$ where $i_0$ is the injective radius at $p$, then
\begin{equation}\label{2}\frac{1}{A(\partial B_p(r))}\int_{\partial B_p(r)} \Delta r_M \leq \Delta r_N(r) - \epsilon\end{equation}
where $\epsilon$ is a positive constant depending only on $m$, $k$ and $r$. In particular, if $p$ is a pole, then (\ref{2}) holds for any $r>0$. In this case, if $r \geq c > 0$, then there exists a constant $\delta > 0$ depending only on $m, k, c$ such that $\epsilon > \delta > 0$. }
\end{theorem}

When the metric is unitary invariant with respect to a point, we have the sharp Laplacian comparison.

\begin{theorem}\label{thm3}
{Let $M^m(m>1)$ be a complete K\"ahler manifold with $Ric \geq (m+1)k$ and suppose the metric is unitary invariant with respect to $p$ in $M$. Let $M_k$ be the complex space form with holomorphic bisectional curvature $k$. Denote $r_M(x)$ to be distance function from $p$ to $x$ in $M$. Let $r_{M_k}$ be the distance function on $M_k$. Then for any $x \in M$,  $y \in M_k$ with $r_M(x) = r_{M_k}(y)$,  $$\Delta r_M(x) \leq \Delta r_{M_k}(y).$$}
\end{theorem}
\begin{remark}
{It is shown in \cite{[L]} that in general, the sharp Laplacian comparison does not hold comparing with the complex space forms.}
\end{remark}

Finally, we have the counterpart of Yau's gradient estimate \cite{[Y]} on K\"ahler manifolds:
\begin{theorem}\label{thm4}
{ Let $M^m(m>1)$ be a complete K\"ahler manifold with $Ric \geq -(2m-1)$. If $f$ is a positive harmonic function on $M$, then \begin{equation}|\nabla \log f| \leq 2m-1-\epsilon\end{equation} where $\epsilon$ is a positive constant depending only on $m$.}
\end{theorem}

\begin{remark}
{Yau's gradient estimate is sharp in the Riemannian case, see \cite{[LW1]}.}
\end{remark}

We have organized this paper into five parts apart from the introduction.
Section 2 is devoted to establishing the Bochner type formula (\ref{1.1}) we will need in the sequel.
We prove theorem \ref{thm1} and its corollaries, as well as theorem \ref{thm2} in section 3.
In section 4, we give the proof to theorem \ref{thm3}. We also prove a sharp average Laplacian comparison theorem under a condition slightly weaker than theorem \ref{thm3}. An example is given in section 5 to show that if the Ricci curvature has a positive lower bound, the sharp version of theorem \ref{thm1} does not hold comparing with complex space forms. We shall compare the example with the result in \cite{[LW]} by Li and Wang, as well as the local results in \cite{[L]}.
The proof of theorem \ref{thm4} is given in the last section.

Here are some notations in this paper.
We shall use Einstein summation in this paper. For a smooth function $f$ on a manifold $M$, $\Delta f$ denotes the standard Beltrami Laplacian if we use orthonormal frame; if we use unitary frame, then $\Delta f = f_{\alpha\overline{\beta}}g^{\alpha\overline\beta}$ which is one half of the Beltrami Laplacian.
For $p \in M$, $B_p(r)$ denotes the geodesic ball in $M$ centered at $p$ with radius $r$. $Vol$ denotes the volume and $A$ denotes the area. Given a compact set $K\in M$, $\dashint_Kf$ is the average of the integral of $f$ over $K$.

\bigskip
\vskip.1in
\begin{center}
\bf  {\quad Acknowledgment}
\end{center}

The author would like to express his deep gratitude to his advisor, Professor Jiaping Wang for constant help and many valuable discussions during the work. He also thanks Professor Peter Li for his interest in this work.

\section{\bf{ A Bochner type formula for functions on K\"ahler manifolds}}
\begin{prop}\label{prop1}
{Let $M^m$($m >1$) be a complete K\"ahler manifold, $m = dim_{\mathbb{C}}(M)$. Let $f\in C^{\infty}(M)$ and assume that $\nabla f(p) \neq 0$ where $p\in M$. Choosing a unitary frame $e_\alpha \in T^{1,0}(M)(\alpha = 1, 2,..m)$ near $p$ so that $e_1 = \frac{1}{\sqrt{2}}(X-\sqrt{-1}JX)$ where $X = \frac{\nabla f}{|\nabla f|}$,
we have
\begin{equation}\label{1.1}
\begin{aligned}
\frac{1}{2}\langle \nabla f, \nabla (\sum\limits_{\gamma \neq 1}f_{\gamma \overline{\gamma}})\rangle = f_{1\overline{1}}\Delta f -|f_{\alpha \overline{\beta}}|^2+ Re (div Y)
 \end{aligned}
\end{equation}}
where $Y =\sum\limits_{\gamma \neq 1}f_{\overline{\alpha}}f_{\alpha \overline{\gamma}}e_{\gamma}$, $\Delta f = \sum\limits_{\alpha}f_{\alpha\overline\alpha}$.
\end{prop}
\begin{proof}

Recall the Bochner formula:
 \begin{equation}\label{2.1}
\frac{1}{2}\Delta(|\nabla f|^2) = |f_{\alpha \beta}|^2 + |f_{\alpha \overline{\beta}}|^2 + (\Delta f)_{\alpha} f_{\overline{\alpha}} + (\Delta f)_{\overline{\alpha}}f_{\alpha} + Ric_{\alpha \overline{\beta}}f_{\overline{\alpha}}f_{\beta}.
\end{equation}
(\ref{2.1}) can be decomposed into two parts, namely,
\begin{equation}\label{2.2}
(f_{\overline{\alpha}}f_{\alpha \overline{\beta}})_{\beta} = |f_{\alpha \overline{\beta}}|^2+(\Delta f)_{\alpha} f_{\overline{\alpha}},
\end{equation}
\begin{equation}\label{2.3}
(f_{\alpha}f_{\overline{\alpha}\overline{\beta}})_{\beta}= |f_{\alpha \beta}|^2+(\Delta f)_{\overline{\alpha}}f_{\alpha} + Ric_{\alpha \overline{\beta}}f_{\overline{\alpha}}f_{\beta}.
\end{equation}

Define a vector field $$Z = f_{\overline{\alpha}}f_{\alpha \overline{1}}e_1,$$
then (\ref{2.2}) becomes
\begin{equation}\label{2.4}
div Y + div Z = |f_{\alpha \overline{\beta}}|^2+(\Delta f)_{\alpha} f_{\overline{\alpha}}.
\end{equation}
Now we compute
 \begin{equation}\label{2.5}
\begin{aligned}
Re (div Z) &= Re (\sum\limits_{\beta \neq 1}\langle \nabla_{e_{\beta}} (f_{\overline{\alpha}}f_{\alpha \overline{1}}e_1), e_{\overline{\beta}}\rangle + \langle \nabla_{e_1}(f_{\overline{\alpha}}f_{\alpha \overline{1}}e_1), e_{\overline{1}}\rangle) \\&= Re(\sum\limits_{\beta \neq 1}f_{\alpha\overline{1}}\langle \nabla_{e_{\beta}} (f_{\overline{\alpha}}e_1), e_{\overline\beta}\rangle + f_{\alpha\overline{1}}\langle \nabla_{e_1}(f_{\overline\alpha}e_1) , e_{\overline 1}\rangle + e_1(f_{\alpha\overline{1}})\langle f_{\overline\alpha}e_1, e_{\overline 1}\rangle)\\&=f_{1\overline{1}}\Delta f +\frac{1}{2}\langle \nabla f, \nabla (f_{1\overline{1}})\rangle.
 \end{aligned}
\end{equation}

Plugging (\ref{2.5}) in (\ref{2.4}), we find
\begin{displaymath}
\begin{aligned}
\frac{1}{2}\langle \nabla f, \nabla (\sum\limits_{\gamma \neq 1}f_{\gamma \overline{\gamma}})\rangle = f_{1\overline{1}}\Delta f -|f_{\alpha \overline{\beta}}|^2+ Re (div Y).
 \end{aligned}
\end{displaymath}

This completes the proof of proposition \ref{prop1}.
\end{proof}
\begin{remark} Note that in (\ref{1.1}), it is assumed that $\nabla f \neq 0$ at $p$. In some applications, we will multiply (\ref{1.1}) on both side by cut-off functions and do integration by parts. We can justify the integration by approximation of Morse functions, no matter whether $\nabla f$ is vanishing somewhere.
\end{remark}

\section{\bf{Relative volume comparison}}

In this section we are going to prove theorem \ref{thm1} and its corollaries, together with theorem \ref{thm2}.
First we shall prove the corollaries in the introduction assuming theorem \ref{thm1}.

\noindent\emph{Proof of corollary \ref{cor1}:}
Suppose for sufficiently small $\epsilon$, \begin{equation}\label{3.1}
\frac{Vol(B_p(b))}{Vol(B_p(a))} \geq \frac{Vol(B_N(b)) }{Vol(B_N(a))}(1 - \epsilon).\end{equation}

We have
\begin{equation}\label{3.0}
\begin{aligned}
\frac{Vol(B_p(b))}{Vol(B_p(a))}&= \frac{Vol(B_p(\frac{a+b}{2}))}{Vol(B_p(a))} + \frac{A(\partial(B_p(\frac{a+b}{2})))}{Vol(B_p(a))}\int\limits_{\frac{a+b}{2}}^{b}\frac{A(\partial(B_p(r)))}{A(\partial(B_p(\frac{a+b}{2})))}
dr \\&\leq \frac{Vol(B_N(\frac{a+b}{2}))}{Vol(B_N(a))} +\frac{A(\partial(B_p(\frac{a+b}{2})))}{Vol(B_p(a))}\int\limits_{\frac{a+b}{2}}^{b}\frac{A(\partial(B_N(r)))}{A(\partial(B_N(\frac{a+b}{2})))}dr.
\end{aligned}
\end{equation}
Putting (\ref{3.0}), (\ref{3.1}) together, after some manipulation, we find
\begin{equation}\label{3.2}
\frac{A(\partial(B_p(\frac{a+b}{2})))}{Vol(B_p(a))} \geq \frac{A(\partial(B_N(\frac{a+b}{2})))}{Vol(B_N(a))}(1-\delta_1).\end{equation}
Also note that
\begin{equation}\label{3.-1}
\begin{aligned}
\frac{Vol(B_p(a))}{A(\partial(B_p(\frac{a+b}{2})))} &=
\frac{A(\partial(B_p(a)))}{A(\partial(B_p(\frac{a+b}{2})))}\int\limits_{0}^{a}
\frac{A(\partial(B_p(r)))}{A(\partial(B_p(a)))}dr \\&\geq
\frac{A(\partial(B_p(a)))}{A(\partial(B_p(\frac{a+b}{2})))}\int\limits_{0}^{a}
\frac{A(\partial(B_N(r)))}{A(\partial(B_N(a)))}dr.
\end{aligned}
\end{equation}
Combining (\ref{3.2}), (\ref{3.-1}) together, we get
\begin{equation}\label{3.3}
\frac{A(\partial(B_p(\frac{a+b}{2})))}{A(\partial(B_p(a)))} \geq \frac{A(\partial(B_N(\frac{a+b}{2})))}{A(\partial(B_N(a)))}(1-\delta_2).\end{equation}
In (\ref{3.2}), (\ref{3.3}), $\delta_1, \delta_2$ are positive constants depending only on $\epsilon, a, b, m, k$.
Moreover, $\lim\limits_{\epsilon \to 0}\delta_i = 0$ for $i = 1, 2$. If $k = -1$, $\delta_i$ depends only on $\epsilon, b-a, c, m$.

If $\epsilon$ is sufficiently small, (\ref{3.3})
contradicts theorem \ref{thm1}.           \qed

\bigskip

\noindent\emph{Proof of corollary \ref{cor2}:}
Let $N$ be the $2m$ dimensional real space form with constant sectional curvature $-1$.
Taking $a_i = i, b_i = i + 1$ in corollary \ref{cor1} for $i = 1, 2,....$, we have
\begin{equation}\label{3.-2}
\frac{Vol(B_p(i+1))}{Vol(B_p(i))} \leq (1-\epsilon_i)\frac{Vol(B_N(i+1))}{Vol(B_N(i))}.
\end{equation}
According to corollary \ref{cor1}, there exists a positive constant $\delta$ such that $\epsilon_i > \delta$ for all $i \geq 1$. Therefore (\ref{3.-2}) becomes
\begin{equation}\label{3.-3}
\frac{Vol(B_p(i+1))}{Vol(B_p(i))} \leq (1-\delta)\frac{Vol(B_N(i+1))}{Vol(B_N(i))}.
\end{equation}
By iteration of (\ref{3.-3}), it follows that
\begin{equation}
\frac{Vol(B_p(i))}{Vol(B_p(1))} \leq (1-\delta)^{i-1}\frac{Vol(B_N(i))}{Vol(B_N(1))}.
\end{equation}
Thus
\begin{equation}\label{3.-4}
\begin{aligned}
\frac{\ln Vol(B_p(i))}{i} &\leq \frac{i-1}{i}\ln(1-\delta) + \frac{\ln Vol(B_N(i))}{i} \\&+ \frac{\ln {Vol(B_p(1))}}{i} -\frac{Vol(B_N(1))}{i}.
\end{aligned}
\end{equation}

When $i\to \infty$, the RHS of (\ref{3.-4}) is approaching $2m-1 + \ln(1-\delta)$.
 This completes the proof of corollary \ref{cor2}.  \qed

\bigskip

\noindent\emph{Proof of corollary \ref{cor3}:}
Let $S^{2m}$ be the $2m$ dimensional sphere with constant sectional curvature $1$. Assuming $d(M) = d$, we pick two points $p, q \in M$ such that $dist(p, q) = d(M)$. According to corollary \ref{cor1}, there exists a positive constant $\epsilon$ such that
$$\frac{Vol(B_p(\frac{d}{2}))}{Vol(B_p(d))} \geq \frac{Vol( B_{S^{2m}}(\frac{d}{2})) }{Vol( B_{S^{2m}}(d))}(1 + \epsilon),$$
$$\frac{Vol(B_q(\frac{d}{2}))}{Vol(B_q(d))} \geq \frac{Vol( B_{S^{2m}}(\frac{d}{2}) )}{Vol( B_{S^{2m}}(d))}(1 +\epsilon).$$

Therefore
\begin{equation}\label{3.4}
\begin{aligned}
1 &\geq \frac{Vol( B_p(\frac{d}{2}))+Vol(B_q(\frac{d}{2}))}{Vol(M)} \\&\geq 2(1+\epsilon)\frac{Vol(B_{S^{2m}}(\frac{d}{2}))} {Vol(B_{S^{2m}}(d))}.
\end{aligned}
\end{equation}

If $d$ is sufficiently close to $\pi$, the right hand side of (\ref{3.4}) is greater than 1. This is a contradiction.  \qed

\begin{remark}{The counterexample in section 5 shows that when $Ric \geq 2m-1$, the diameter of the K\"ahler manifold could exceed that of $\mathbb{C}\mathbb{P}^m$. The corollary says the diameter of the K\"ahler manifold can not be too close to that of $S^{2m}$.}
\end{remark}

\noindent\emph{Proof of corollary \ref{cor4}:}

We use the same notation as in theorem \ref{thm1}. Denote the area of the geodesic sphere $\partial B_p(r)$ by $A(r)$, the volume of the geodesic ball $B_p(r)$ by $V(r)$. Denote $\lambda_1(B_N( r))$ by $\lambda_1$ and let
 $f$ be the nonnegative eigenfunction to the equation $$\Delta f = -\lambda_1 f$$ on $B_{N}(r)$ with Dirichlet boundary condition. After normalization, we may assume $\int_{B_{N}( r)} f^2 = 1$. It is easy to see that $f$ is a radical function. Pulling $f$ back to the tangent space of $p$, via the exponential map, we may assume that $f$ is defined on $B_p(r)$.

Suppose there is small constant $\epsilon$ such that $$\lambda_1(B_p(r)) \geq \lambda_1 - \epsilon,$$  then we have the inequality
\begin{equation}\label{3.5}
\lambda_1 - \epsilon \leq \frac{\int_{B_p(r)}|\nabla f|^2}{\int_{B_p(r)} f^2}.\end{equation}
Using integration by parts, we find
\begin{equation}\label{3.6}
\frac{\int_{B(P, r)}(\lambda_1 f + \Delta f)f}{\int_{B_p(r)}f^2} \leq \epsilon.
\end{equation}
By Cheng's argument in \cite{[C]}, $$\lambda_1 f + \Delta f \geq 0$$ in $B_p(r)$.
It is simple to see that $f$ is strictly between two positive constants in $B_p(\frac{r}{2})$.
By (\ref{3.6}), we have
\begin{equation}\label{3.-5}
\begin{aligned}
(\min\limits_{B_p(\frac{r}{2})}f) \frac{\int_{B_p(\frac{r}{2})}(\lambda_1 f + \Delta f)}{V(\frac{r}{2})}&\leq
\frac{\int_{B_p(\frac{r}{2})}f(\lambda_1 f + \Delta f)}{V(\frac{r}{2})} \\&\leq
\frac{\int_{B_p(r)}f(\lambda_1 f + \Delta f)}{V(\frac{r}{2})}\\&\leq \epsilon \frac{\int_{B_p(r)}f^2}{V(\frac{r}{2})}
\\&\leq \epsilon (\max\limits_{B_p(r)}f^2)\frac{V(r)}{V(\frac{r}{2})}\\&\leq C(r, k, m)\epsilon\max\limits_{B_p(r)}f^2.
\end{aligned}
\end{equation}
Therefore, we conclude
\begin{equation}\label{3.7}
\begin{aligned}
\dashint_{B(P, \frac{r}{2})}(\lambda_1 f + \Delta f) &\leq C(r, k, m)\epsilon\frac{\max\limits_{B_p(r)}f^2}{\min\limits_{B_p(\frac{r}{2})}f}\\&=\delta(\epsilon, r, k, m).
\end{aligned}
\end{equation}
 Noting that $f$ is a function of $r$ and $f' \leq 0$, we have
\begin{equation}\label{3.8}
\begin{aligned}\dashint_{B(P, \frac{r}{2})}\lambda_1 f & = \lambda_1\frac{\int^{\frac{r}{2}}_{0}f(t)A(t)dt}{V(\frac{r}{2})} \\&= \lambda_1 f(\frac{r}{2}) + \lambda_1\int^{\frac{r}{2}}_{0} (-f'(t))\frac{V(t)}{V(\frac{r}{2})}dt \\&\geq C(r)\end{aligned}\end{equation} where $$C(r)= \lambda_1 f(\frac{r}{2}) + \lambda_1\int^{\frac{r}{2}}_{0} (-f'(t))\frac{Vol_N(t)}{Vol_N(\frac{r}{2})}dt.$$
In the last inequality of (\ref{3.8}), we have applied the Bishop-Gromov volume comparison.
Using the divergence theorem, we have \begin{equation}\label{3.9}
\dashint_{B(P, \frac{r}{2})} \Delta f = f'(\frac{r}{2})\frac{A(\frac{r}{2})}{V(\frac{r}{2})}.
\end{equation}

Combining (\ref{3.7}), (\ref{3.8}), (\ref{3.9}), we obtain
\begin{equation}\label{3.-14}
\frac{A(\frac{r}{2})}{V(\frac{r}{2})} \geq \frac{C(r)- \delta}{(-f'(\frac{r}{2}))} = \frac{A(B_N(\frac{r}{2}))}{Vol(B_N(\frac{r}{2}))} - \overline{\delta}
\end{equation}
where $\delta, \overline\delta$ are small constants depending on $\epsilon, m, k, r$.

If $\epsilon$ is very small, $\overline\delta$ is small.
(\ref{3.-14}) contradicts theorem \ref{thm1}.               \qed

\bigskip
\noindent\emph{Proof of theorem \ref{thm1}:}

We consider the case $Ric \geq -(2m-1)$ first.

Let $n = 2m$. For $x\in M$, define $r(x) = d(x, p)$.  Choose an orthonormal frame $h_i$ ($i = 1, 2,.. 2m$) near $x$ so that $h_1 = \nabla r$ and $Jh_{2\alpha-1}=h_{2\alpha}$ for $1\leq \alpha \leq m$. Define a unitary frame $\{e_{\alpha}\}$ so that $e_{\alpha} = \frac{1}{\sqrt{2}}(h_{2\alpha-1} - \sqrt{-1}h_{2\alpha})$. Let $\omega^i$ be the dual 1-form of $h_i$. Define a tensor $S$ near $x$ such that
\begin{equation}\label{3.10}
\begin{aligned}
S &= S_{ij}\omega^i\otimes\omega^j\\& = \coth r\sum\limits_{i\neq 1}\omega^i\otimes\omega^i.
\end{aligned}
\end{equation}

It is simple to see that the tensor $S$ is independent of the frame $h_i$, moreover, $S$ is the Hessian of the distance function in real space form with sectional curvature $-1$. After the complexification, we find
\begin{equation}\label{3.-10}
S_{\alpha\overline\beta} = \left\{
\begin{array}{rl}  0 &  \alpha \neq \beta \\ \coth r & \alpha=\beta, \alpha \neq 1 \\
\frac{1}{2}\coth r & \alpha = \beta = 1 .\\
\end{array} \right.
\end{equation}

We introduce the proposition as follows:
\begin{prop}\label{prop3}
{Let $M^n$ be a complete Riemannian manifold such that $Ric \geq -(n-1)$, $p \in M$ be a point.
Define $N$ to be the $n$ dimensional real space form with constant sectional curvature $-1$.
Given constants $b>a > c > 0$, $\epsilon > 0$, if the area of the geodesic spheres satisfies
\begin{equation}\label{3.11}
\frac{A(\partial B_p(b))}{A(\partial B_p(a))} \geq \frac{A(\partial B_N(b))}{A{(\partial B_N(a))}}- \epsilon,\end{equation}
 there exists positive constants $\delta$, $C$ and a smooth function $w$ defined in the annulus $T = \{x\in M| \frac{3a+2b}{5} \leq d(x, p) \leq \frac{2a+3b}{5}\}$ so that
\begin{equation}\label{3.-9}
\begin{aligned}
\dashint_{T}(|\nabla w - \nabla r|^2 &+ \sum\limits_{i, j}|w_{ij} - S_{ij}|^2) < \delta(b-a, c, n, \epsilon), \\&|\nabla w| < C(b-a, c, n).
\end{aligned}
\end{equation}
Moreover, $$\lim\limits_{\epsilon \to 0}\delta(b-a, n, c, \epsilon) = 0.$$}
\end{prop}

\begin{remark}
Proposition \ref{prop3} originates from Cheeger and Colding's paper \cite{[CC]}. Their estimate depends on both the upper bound and lower bound of $a$ and $b$ which is not sufficient to prove corollary 2.
\end{remark}
 \begin{proof}

For notational convenience, in the proof of proposition \ref{prop3}, $\delta$ denotes small positive constants depending only on $\epsilon, c, b-a, n$. $C$ denotes positive constants depending only on $c, b-a, n$. Moreover, $\lim\limits_{\epsilon\to 0}\delta = 0$.

Define $\overline\Delta, \overline\nabla$ to be the Laplacian and the covariant derivatives in $N$.
Pick a point $\overline{p}$ in $N$, define
$$A(a, b) = \{x\in M|a\leq d(x, p)\leq b\},$$
$$\overline{A}(a,b) =
\{y\in N|a\leq d(y, \overline p)\leq b\}.$$

We solve the equation
\begin{equation}\label{3.-15}
\overline\Delta \overline{g} = 0
 \end{equation}
in $\overline{A}(a,b)$ satisfying the boundary condition $\overline g(x) = 2$ on $\partial B(\overline p, a)$, $\overline g(x) = 1$ on $\partial B(\overline P, b)$.  Then $\overline{g}$ is a radical function, say, $\overline g = \phi(r)$.  Pulling $\overline{g}$ back to the tangent space of $p$, via the exponential map, we may assume that $\overline{g}$ is defined in $A(a, b)$. It is simple to check that $\overline{g}$ is strictly decreasing with respect to $r$, therefore $\overline g^{-1}$ exists. Moreover,
\begin{equation}\label{3.-6}
\overline{g},|\overline{g}'|, |\overline{g}''|, |(\overline{g}^{-1})'|, |(\overline{g}^{-1})''|  < C.
\end{equation}

Similarly we solve the equation
\begin{equation}\label{3.-16}
\Delta g = 0
\end{equation}
in $A(a, b)$ satisfying the same boundary condition as $\overline g$.

\begin{claim}\label{cl1}
\begin{equation}
\begin{aligned}
&|\frac{A(\partial B_p(a))}{Vol(A(a, b))} - \frac{A(\partial B_{\overline p}(a))}{Vol(\overline A(a, b))}| < \delta,\\& |\frac{A(\partial B_p(b))}{Vol(A(a, b))} - \frac{A(\partial B_{\overline p}(b))}{Vol(\overline A(a, b))}| < \delta.
\end{aligned}
\end{equation}
\end{claim}
\begin{proof}
Claim \ref{cl1} follows from (\ref{3.11}) and Bishop-Gromov volume comparison.
\end{proof}

The following claim is due to J. Cheeger and T. Colding \cite{[CC]}:
\begin{claim}\label{cl2}
\begin{equation}\dashint_{A(a,b)}|\nabla g - \nabla \overline{g}|^2 \leq \delta.
 \end{equation}
\end{claim}
\begin{proof}
By maximum principle, $|g-\overline g|\leq 2$ in $A(a, b)$. By Laplacian comparison and (\ref{3.-15}), we have $\Delta \overline g \geq 0$, since $\overline g$ is decreasing with respect to $r$.
Using integration by parts, we have
\begin{displaymath}
\begin{aligned}
\dashint_{A(a,b)}|\nabla g - \nabla \overline{g}|^2& = -\dashint_{A(a,b)}(g - \overline{g})\Delta(g-\overline{g})\\&
\leq 2\dashint_{A(a,b)} \Delta \overline g \\&\leq 2\frac{\int_{\partial A(a,b)}\frac{\partial \overline g}{\partial n}ds}{Vol(A(a,b))} \\& = 2(\overline g'(b)\frac{A(\partial B_p(b))}{Vol(A(a, b))} - \overline g'(a)\frac{A(\partial B_p(a))}{Vol(A(a, b))})\\&\leq \delta.
\end{aligned}
\end{displaymath}
In the last inequality, we have applied claim \ref{cl1}.
\end{proof}

The claim below comes from P. Li and R. Schoen in \cite{[LS]} and J. Cheeger and T. Colding \cite{[CC]}:
\begin{claim}\label{cl3}
For Dirichlet boundary condition on $A(a,b)$, the first eigenvalue $\lambda_1 \geq C$.
\end{claim}
\begin{proof}
Given a constant $\mu > 0$, define $f = (b+1-r)^{\mu}$ in $A(a,b)$. Then $$\Delta f = \mu(\mu-1)(b+1-r)^{\mu-2}-\mu(b+1-r)^{\mu-1}\Delta r,$$ $$|\nabla f| = \mu(b+1-r)^{\mu -1}.$$

Taking $\mu$ large, from the Laplacian comparison theorem, we may assume that
\begin{displaymath}
\Delta f \geq 1, |\nabla f| \leq C
\end{displaymath}
in $A(a, b)$.

For any function $h\in C^{\infty}_{c}(A(a,b))$,
\begin{displaymath}
\begin{aligned}
\int_{A(a,b)}h^2 &\leq \int_{A(a,b)}h^2\Delta f \\& = -2\int_{A(a,b)}h\langle \nabla f, \nabla h\rangle \\&
\leq 2C\int _{A(a,b)}|h||\nabla h| \\& \leq 2C (\int_{A(a,b)}h^2)^{\frac{1}{2}}(\int_{A(a,b)}|\nabla h|^2)^{\frac{1}{2}}
\end{aligned}
\end{displaymath}

Therefore $\lambda_1 \geq \frac{1}{4C^2}$. This proves claim \ref{cl3}.
\end{proof}
Combining claim \ref{cl2} and claim \ref{cl3}, one concludes that
\begin{equation}\label{3.12}
\dashint_{A(a,b)}|g-\overline{g}|^2 \leq \delta.
\end{equation}

Let $a_1 = \frac{4a+b}{5}, a_2=\frac{3a+2b}{5}, b_2=\frac{2a+3b}{5}, b_1=\frac{a+4b}{5}$, so $a < a_1 < a_2 < b_2< b_1 <b$.
Since $g$ is a positive harmonic function in $A(a,b)$, by the gradient estimate of Cheng and Yau \cite{[CY]}, we have
\begin{equation}\label{3.13}
|\nabla g| \leq Cg \leq 2C
\end{equation}
in $A(a_1, b_1)$. By a simple calculation, one can find a function $\psi(r)$ so that
$$\overline\nabla^2\psi(r) = \frac{1}{n}\overline\Delta(\psi(r)),$$ $$\psi(a) = 1, \psi'(a) = 1$$ in $\overline A(a,b)$. It is easy to see $\psi$ is strictly increasing with respect to $r$, therefore $\psi^{-1}$ exists. Moreover,
\begin{equation}\label{3.14}
|\psi|, |\psi'|, |\psi''|, |(\psi^{-1})'|, |(\psi^{-1})''| , |(\psi^{-1})'''|< C.
\end{equation}

For $x \in A_{a, b}$, define
\begin{equation}\label{3.-8}
u(x) = \psi\circ\overline{g}^{-1}\circ g(x), \overline u(x) = \psi\circ\overline{g}^{-1}\circ \overline g(x) = \psi(r(x)),
\end{equation}
then by (\ref{3.-6}), (\ref{3.13}), (\ref{3.14}),
\begin{equation}\label{3.15}
|\nabla u| \leq C
\end{equation}
in $A(a_1, b_1)$.

The Bochner formula for $u(x)$ is
$$\frac{1}{2}\Delta |\nabla u|^2 = |\nabla^2 u|^2 + \langle \nabla u, \nabla \Delta u\rangle + Ric(\nabla u, \nabla u).$$

Since $Ric \geq -(n-1)$,  we can rewrite it as
\begin{equation}\label{3.16}
\frac{1}{2}\Delta |\nabla u|^2 - \langle \nabla \Delta u, \nabla u\rangle - \frac{1}{n}(\Delta u)^2 +(n-1)\langle \nabla u, \nabla u\rangle \geq |\nabla^2 u - \frac{1}{n}\Delta u|^2.
\end{equation}

Now we want to get the estimate of the Hessian of $u$.

We will multiply both side of (\ref{3.16}) by a cut-off function and do integration by parts.

In \cite{[CC]}, Cheeger and Colding choose the cut-off function to be a function of $g$. To make that work, the  upper bound of $a$ is needed. To avoid this problem, we define the cut-off function $\overline\varphi$ to be a function of $r$, explicitly,
 \begin{equation}\label{3.17}
 \overline\varphi(r) = \left\{\begin{array}{rl} 0 & a \leq r < a_1 \\
 \frac{r-a_1}{a_2-a_1}  & a_1\leq r \leq a_2 \\
 1 & a_2 \leq r \leq b_2 \\
 \frac{b_1 - r}{b_1 - b_2} & b_2 < r < b_1 \\
 0 & b_1 \leq r \leq b. \end{array}\right.
 \end{equation}

Define \begin{equation}\label{3.18}
\varphi(x) = \overline\varphi\circ\overline{g}^{-1}\circ g
\end{equation} in $A(a, b)$. From claim \ref{cl2}, (\ref{3.-6}), (\ref{3.12}), (\ref{3.13}), it is easy to see that
\begin{equation}\label{3.-7}
\dashint_{A(a,b)} |\nabla \overline \varphi - \nabla \varphi|^2 +|\varphi - \overline\varphi|^2 \leq \delta.\end{equation}

Multiplying (\ref{3.16}) on both side by $\overline\varphi^2$ and using integration by parts, we find
\begin{equation}\label{3.19}
\begin{aligned}
\frac{1}{Vol(A(a_1, b_1))}\int_{A(a_1, b_1)}&-\frac{1}{2}\langle \nabla \overline\varphi^2, \nabla |\nabla u|^2\rangle + \overline\varphi^2 (\Delta u)^2 +
2\overline\varphi\Delta u \langle \nabla \overline\varphi, \nabla u\rangle - \frac{1}{n}\overline\varphi^2 (\Delta u)^2 \\&+(n-1)\overline\varphi^2 |\nabla u|^2 \geq
\frac{1}{Vol(A(a_1, b_1))}\int_{A(a_1, b_1)}\overline\varphi^2|\nabla^2 u - \frac{1}{n}\Delta u|^2.
\end{aligned}
\end{equation}

Let us write the first term of (\ref{3.19}) as
\begin{equation}\label{3.20}
\begin{aligned}
-\frac{1}{Vol(A(a_1, b_1))}\int_{A(a_1, b_1)}\frac{1}{2}\langle \nabla \overline\varphi^2, \nabla |\nabla u|^2\rangle &= -\frac{1}{2}\dashint_{A(a_1, b_1)}
(\overline\varphi^2)_i(u^2_j)_i \\& = -2\dashint_{A(a_1, b_1)}\overline\varphi\overline\varphi_iu_ju_{ji}\\& =
-2(\dashint_{A(a_1, b_1)}\overline\varphi(\overline\varphi_i- \varphi_i)u_ju_{ji} + \dashint_{A(a_1, b_1)}\overline\varphi\varphi_iu_ju_{ji}).
\end{aligned}
\end{equation}

We will estimate the two terms in the RHS of (\ref{3.20}) separately.
Using (\ref{3.15}), (\ref{3.-7}), we find
\begin{equation}\label{3.21}
\begin{aligned}
&|-2\dashint_{A(a_1, b_1)}\overline\varphi(\overline\varphi_i-\varphi_i)u_ju_{ji}| \\&\leq 2\dashint_{A(a_1, b_1)}|\overline\varphi||\nabla\overline\varphi-\nabla\varphi||\nabla u||u_{ji}| \\& \leq C\dashint_{A(a_1, b_1)} \overline\varphi|\nabla \varphi - \nabla \overline \varphi||\nabla^2 u| \\& \leq C(\dashint_{A(a_1, b_1)}|\nabla \overline \varphi - \nabla \varphi|^2)^{\frac{1}{2}}(\dashint_{A(a_1, b_1)}\overline\varphi^2|\nabla^2 u|^2)^{\frac{1}{2}}\\&\leq \delta (\dashint_{A(a_1, b_1)}\overline\varphi^2|\nabla^2 u|^2)^{\frac{1}{2}}\\& \leq \frac{1}{2}(\delta + \delta \dashint_{A(a_1,b_1)}\overline\varphi^2|\nabla^2 u|^2).
 \end{aligned}
\end{equation}

In (\ref{3.21}) we may assume that $\epsilon$ is so small that $\delta \leq \frac{1}{2}$.
For the other term on the RHS of (\ref{3.20}), integration by parts gives
\begin{equation}\label{3.22}
\begin{aligned}
-2\dashint_{A(a_1, b_1)}\overline\varphi\varphi_iu_ju_{ji} &= -\dashint_{A(a_1, b_1)}\overline\varphi \varphi_i(u^2_j)_i \\& = \dashint_{A(a_1, b_1)} \overline\varphi_i\varphi_iu^2_j + \dashint_{A(a_1, b_1)}\overline\varphi\varphi_{ii}u^2_j \\& = \dashint_{A(a_1, b_1)} \langle \nabla \varphi, \nabla \overline \varphi\rangle |\nabla u|^2 + \overline\varphi\Delta\varphi |\nabla u|^2.
 \end{aligned}
\end{equation}

Plugging (\ref{3.21}) and (\ref{3.22}) in (\ref{3.20}), we find that
\begin{equation}\label{3.23}
\begin{aligned}
&\dashint_{A(a_1, b_1)} \{\langle \nabla \overline\varphi, \nabla \varphi\rangle |\nabla u|^2 + \overline\varphi\Delta\varphi |\nabla u|^2 + \frac{1}{2}\delta + \frac{1}{2}\delta \overline\varphi^2|\nabla^2 u|^2+\overline\varphi^2 (\Delta u)^2 \\&+
2\overline\varphi\Delta u \langle \nabla \overline\varphi, \nabla u\rangle - \frac{1}{n}\overline\varphi^2 (\Delta u)^2 +(n-1)\overline\varphi^2 |\nabla u|^2\} \\&\geq
\dashint_{A(a_1, b_1)}\overline\varphi^2|\nabla^2 u - \frac{1}{n}\Delta u|^2.
\end{aligned}
\end{equation}

We can rewrite (\ref{3.23}) as
\begin{equation}\label{3.24}
\begin{aligned}
&\dashint_{A(a_1, b_1)} \{\langle \nabla \varphi, \nabla \overline \varphi\rangle |\nabla u|^2 + \overline\varphi\Delta\varphi |\nabla u|^2 + \frac{1}{2}\delta+ \frac{\delta}{2n}\overline\varphi^2(\Delta u)^2+\overline\varphi^2 (\Delta u)^2 \\&+
2\overline\varphi\Delta u \langle \nabla \overline\varphi, \nabla u\rangle - \frac{1}{n}\overline\varphi^2 (\Delta u)^2 +(n-1)\overline\varphi^2 |\nabla u|^2\} \\&\geq (1-\frac{\delta}{2})\dashint_{A(a_1, b_1)}\overline\varphi^2|\nabla^2 u - \frac{1}{n}\Delta u|^2 \\& \geq
(1-\frac{\delta}{2})\frac{Vol(A(a_2, b_2))}{Vol(A(a_1, b_1))}\dashint_{A(a_2, b_2)}|\nabla^2 u - \frac{1}{n}\Delta u|^2 \\& \geq C\dashint_{A(a_2, b_2)}|\nabla^2 u - \frac{1}{n}\Delta u|^2.
\end{aligned}
\end{equation}

We claim that up to a negligible error, we can replace all functions on the LHS of (\ref{3.24}) by the corresponding radical functions, namely,
$$\nabla\varphi \longrightarrow \nabla \overline\varphi,$$
$$\nabla u \longrightarrow \nabla\overline u,$$
$$\Delta \varphi \longrightarrow (\overline \varphi \circ {\overline g}^{-1})''|\nabla \overline g|^2,$$
$$\Delta u \longrightarrow (\overline \psi \circ {\overline g}^{-1})''|\nabla \overline g|^2.$$

To justify the above substitution we use the following standard lemma:

\begin{lemma} {Let $k_1,..., k_N, \overline{k}_1,..,\overline{k}_N$ be functions on a measure space $U$ such that for all $i$,
$$Sup(|k_i| + |\overline{k}_i|) \leq C,$$ $$\int_U |k_i - \overline{k}_i| \leq \epsilon,$$
then $$\int_U|k_1\cdot\cdot\cdot k_N - \overline{k}_1\cdot\cdot\cdot \overline{k}_N| \leq NC^{N-1}\epsilon.$$}
\end{lemma}
\begin{proof} If we write
\begin{displaymath}
\begin{aligned}
k_1\cdot\cdot\cdot k_N - \overline{k}_1\cdot\cdot\cdot \overline{k}_N & = (k_1 - \overline{k}_1)k_2\cdot\cdot\cdot k_N\\& + \sum\limits_{i=1}^{N-2}\overline{k}_1\cdot\cdot\cdot \overline{k}_i(k_{i+1}-\overline{k}_{i+1})k_{i+2}\cdot\cdot\cdot k_N\\& +
\overline{k}_1\cdot\cdot\cdot \overline{k}_{N-1}(k_N - \overline{k}_N),
\end{aligned}
\end{displaymath}
then the conclusion is obvious.
\end{proof}
It is simply to see that the area of the geodesic sphere in $A(a,b)$ satisfies a pinching estimate:
 \begin{equation}\label{3.25}
 \frac{A(\partial B_N(b))}{A{(\partial B_N(a))}}
\geq \frac{A(\partial B_p(b))}{A(\partial B_p(a))} \geq \frac{A(\partial B_N(b))}{A{(\partial B_N(a))}}- \delta.
\end{equation}

Therefore, after the replacement of functions, the LHS of (\ref{3.24}) is very small.
Thus we have
\begin{equation}\label{3.26}
\dashint_{A(a_2, b_2)}|\nabla^2 u - \frac{1}{n}\Delta u|^2 \leq \delta.
\end{equation}

Since $\Delta g = 0$, it follows from (\ref{3.-6}), (\ref{3.12}), (\ref{3.13}), (\ref{3.14}), (\ref{3.-8}) and claim \ref{cl2} that
\begin{equation}\label{3.27}
\dashint_{A(a_1,b_1)} |\Delta u - (\overline \psi \circ {\overline g}^{-1})''|\nabla \overline g|^2|^2 \leq \delta.
\end{equation}
Also note
$$(\overline \psi \circ {\overline g}^{-1})''|\nabla \overline g|^2 = \overline{\Delta}\overline{u}$$ and
$$\overline{\nabla}^2 \overline u = \frac{1}{n}\overline{\Delta}\overline u,$$
therefore
\begin{equation}\label{3.28}
\dashint_{A(a_2,b_2)} |\nabla^2 u - \overline\nabla^2\overline{u}|^2 \leq \delta.
\end{equation}

Let us define
 \begin{equation}\label{3.29}
 w = \psi^{-1}\circ u = {\overline g}^{-1} \circ g,
\end{equation}

Putting claim \ref{cl2}, (\ref{3.12}), (\ref{3.14}), (\ref{3.15}), and (\ref{3.28}) together, we find
\begin{equation}\label{3.30}
\dashint_{A(a_2,b_2)} \sum\limits_{i,j}|w_{ij} - S_{ij}|^2 \leq \delta
\end{equation}
where we have used the fact that $S$ is the pull back of the Hessian of the distance function from $N$.

It follows from (\ref{3.14}) and (\ref{3.15}) that
\begin{equation}\label{3.31}
|\nabla w| \leq C.
\end{equation}
in $A(a_1, b_1)$.

Putting (\ref{3.-6}) and claim \ref{cl2} together, we find
\begin{equation}\label{3.32}
\dashint_{A(a_1,b_1)} |\nabla w - \nabla r|^2 \leq \delta.
\end{equation}

Combining (\ref{3.30}), (\ref{3.31}) and (\ref{3.32}), we complete the proof of proposition \ref{prop3}.
\end{proof}

Now we are ready to prove theorem \ref{thm1}.
If there exists a small constant $\epsilon > 0$ such that
\begin{equation}\label{3.-11}
\frac{A(\partial B_P(b))}{A(\partial B_P(a))} \geq \frac{A(\partial B_N(b)) }{A(\partial B_N(a)) }(1 - \epsilon),
\end{equation}
according to proposition \ref{prop3}, there exists a smooth function $w$ in $T = A(\frac{3a+2b}{5}, \frac{2a+3b}{5})$ such that (\ref{3.-9}) holds. For simplicity, we write $a_1 = \frac{3a+2b}{5}, b_1 = \frac{2a+3b}{5}, a_2 = \frac{2a_1+b_1}{3}, b_2 = \frac{a_1+2b_1}{3}$(note that our notation is a little bit different from proposition \ref{prop3}).

Applying (\ref{1.1}) to $w$, we find
\begin{equation}\label{3.33}
\frac{1}{2}\langle \nabla w, \nabla (\sum\limits_{\gamma \neq 1}w_{\gamma \overline{\gamma}})\rangle - (w_{1\overline{1}}\Delta w -|w_{\alpha \overline{\beta}}|^2+ Re (div Y))=0
\end{equation}
where $Y = \sum\limits_{\gamma \neq 1}w_{\overline{\alpha}}w_{\alpha \overline{\gamma}}e_{\gamma}$.
Let $\varphi(r)$ to be the Lipschitz cut-off function in the annulus $T$ depending only on the distance to $p$, given by \begin{equation}\label{3.34}
\varphi(r) = \left\{
\begin{array}{rl}\frac{r - a_1}{a_2 - a_1} & a_1 \leq r \leq a_2 \\
1 & a_2 < r < b_2 \\ \frac{b_1 - r}{b_1 - b_2} & b_2 \leq r \leq b_1  .\\
\end{array}\right.
\end{equation}

Multiplying $\varphi$ on both sides of (\ref{3.33}), we integrate in the annulus $T$ and take the average. It follows that
\begin{equation}\label{3.35}
\begin{aligned}
\frac{1}{2}\dashint_T\varphi\langle\nabla w, \nabla(\sum\limits_{\gamma \neq 1}w_{\gamma\overline{\gamma}})\rangle -
\dashint_T\{\varphi w_{1\overline{1}}\Delta w -|w_{\alpha \overline{\beta}}|^2+ \varphi Re (div Y)\} = 0.
\end{aligned}
\end{equation}

Using integration by parts, we find
\begin{equation}\label{3.36}
\begin{aligned}
\frac{1}{2}\dashint_T\varphi\langle\nabla w, \nabla(\sum\limits_{\gamma \neq 1}w_{\gamma\overline{\gamma}})\rangle
&= -\dashint_T\varphi\Delta w\sum\limits_{\gamma \neq 1}w_{\gamma\overline{\gamma}} - \frac{1}{2}\dashint_T\sum\limits_{\gamma \neq 1}w_{\gamma\overline{\gamma}}
\langle\nabla w, \nabla\varphi\rangle,
\end{aligned}
\end{equation}

\begin{equation}\label{3.37}
\begin{aligned}
\dashint_T\varphi (w_{1\overline{1}}\Delta w -|w_{\alpha \overline{\beta}}|^2)+ \varphi Re (div Y)
&=\dashint_T\varphi (w_{1\overline{1}}\Delta w -|w_{\alpha \overline{\beta}}|^2)\\&- Re\dashint_T
\sum\limits_{\gamma \neq 1}\varphi_{\gamma}w_{\overline{\alpha}}w_{\alpha \overline{\gamma}}.
\end{aligned}
\end{equation}
Note that in (\ref{3.36}) and (\ref{3.37}), $\Delta w$ is one half of the real Laplacian of $w$.

Therefore
\begin{equation}\label{3.38}
\begin{aligned}
&-\dashint_T\varphi\Delta w\sum\limits_{\gamma \neq 1}w_{\gamma\overline{\gamma}} - \frac{1}{2}\dashint_T\sum\limits_{\gamma \neq 1}w_{\gamma\overline{\gamma}}\langle\nabla w, \nabla\varphi\rangle-\\&\dashint_T\varphi (w_{1\overline{1}}\Delta w -|w_{\alpha \overline{\beta}}|^2)+ Re\dashint_T
\sum\limits_{\gamma \neq 1}\varphi_{\gamma}w_{\overline{\alpha}}w_{\alpha \overline{\gamma}}=0
\end{aligned}
\end{equation}

Following (\ref{3.-9}), we see that up to a negligible error, we can replace the functions in (\ref{3.38}) by the corresponding functions of $r$. Explicitly,
$$w_{ij} \rightarrow S_{ij}, \nabla w \rightarrow \nabla r.$$

 By (\ref{3.-11}), we can also replace the volume element by that of the real space form $N$.

In order to derive a contradiction to (\ref{3.-11}), we just need to prove that after the replacement, there is an explicit gap between
the LHS and the RHS of (\ref{3.38}). To prove this, it suffices to find a gap between the LHS and RHS of (\ref{1.1}) if we do the replacement:
\begin{equation}\label{3.39}
f_{\alpha\overline\beta}\rightarrow S_{\alpha\overline\beta}, \nabla f \rightarrow \nabla r.
\end{equation}

Using (\ref{3.-10}), we find the gap between the LHS and RHS in (\ref{1.1}) is
\begin{equation}\label{3.-12}
 \begin{aligned}
 &\frac{1}{2}\langle \nabla r, \nabla\sum\limits_{\gamma \neq 1}S_{\gamma \overline{\gamma}}\rangle-(S_{1\overline{1}}S_{\alpha\overline\alpha} -|S_{\alpha \overline{\beta}}|^2+ Re div Y)\\&   = -\frac{m-1}{2}(\coth r)' - \frac{1}{2}\coth r \frac{2m-1}{2}\coth r + (\frac{1}{2}\coth r)^2 + (m-1)(\coth r)^2 \\& = -\frac{m-1}{2}((\coth r)' - (\coth r)^2) \\& = \frac{m-1}{2}> 0.
 \end{aligned}
 \end{equation}

This proves theorem \ref{thm1} for $Ric \geq -(2m-1)$.
The proof for other cases are similar. We complete the proof of theorem \ref{thm1}.
\qed

\bigskip

\noindent\emph{Proof of theorem \ref{thm2}:}
We first prove theorem \ref{thm2} assuming $P$ is a pole, $Ric \geq -(2m-1)$ and $r > 5$. Let $n = 2m$.

For notational convenience, in the proof of theorem \ref{thm2}, $\delta$ denotes small positive constants depending only on $\epsilon, n$. $C$ denotes constants depending only on $n$. Moreover, $\lim\limits_{\epsilon\to 0}\delta = 0$.
We use $\overline r_{ij}$ to denote the Hessian of the distance function on $N$.
\begin{claim}\label{cl4}
{ $1-n \leq \Delta r \leq 100(n-1)$ for $r > 1$}.
\end{claim}
\begin{proof}
By the Laplacian comparison theorem,
$$\Delta r \leq (n-1)\coth r \leq 100(n-1)$$ if $r > 1$. The Bochner formula gives
\begin{equation}\label{3.-13}
\frac{\partial \Delta r}{\partial r} + \frac{1}{n-1}(\Delta r)^2 - n + 1 \leq 0.
\end{equation}
If $\Delta r < 1 - n$, then after a simple ODE analysis, $\Delta r$ will blow up when $r$ is large. This is a contradiction.
\end{proof}
Choose an orthonormal frame $h_i$ ($i = 1, 2,.. 2m$) near $x \in M$ such that $h_1 = \nabla r$ and $Jh_{2\alpha-1}=h_{2\alpha}$ for $1\leq \alpha \leq m$. Define a unitary frame $e_{\alpha}$($\alpha = 1,...m$) so that $e_{\alpha} = \frac{1}{\sqrt{2}}(h_{2\alpha-1} - \sqrt{-1}h_{2\alpha})$ for all $\alpha$.
\begin{claim}\label{cl5}
 {Along the geodesic emanating from $p$, we have $\int\limits_r^{r+1} |r_{ij}|^2 < C$ for any $r \geq 1$.}
\end{claim}
\begin{proof}
According to the Bochner formula,
\begin{equation}\label{3.40}
\frac{\partial \Delta r}{\partial r} + |r_{ij}|^2 -(n-1) \leq 0.
\end{equation}

The result follows from claim \ref{cl4} after we integrate (\ref{3.40}) along the geodesic from $r$ to $r+1$.
\end{proof}
Now we argue by contradiction for theorem \ref{thm2}. Given any $r_0 > 5$, assume the average of the Laplacian satisfies
\begin{equation}\label{3.41}
\dashint_{\partial B_p(r_0)} \Delta r \geq \Delta r_N(r_0) - \epsilon
 \end{equation}
 where $\epsilon$ is a small positive constant,
then $\partial B_p(r)$ can be decomposed into two parts, namely,
\begin{equation}\label{3.42}
\begin{aligned}
E_1 = \{x \in \partial B_p(r_0)| \Delta r \geq \Delta r_N(r_0) - \sqrt \epsilon\},\\
E_2 = \{x \in \partial B_p(r_0)| \Delta r < \Delta r_N(r_0) - \sqrt \epsilon\}.
\end{aligned}
\end{equation}
For $i = 1, 2$, define the cone as follows:
\begin{equation}\label{3.43}
F_i = \{\theta \in UT_p(M)|exp_p(r_0\theta)\in E_i\}.
\end{equation}
We also define
\begin{equation}\label{3.44}
E_i(r) = \{x\in M|x \in exp_p(rF_i)\}.
\end{equation}

\begin{claim}\label{cl6}
\begin{equation}\label{3.45}
\frac{A(E_1)}{A(\partial B_p(r_0))} \geq 1 - \delta.
\end{equation}
\end{claim}
\begin{proof}
From (\ref{3.41}), (\ref{3.42}), we have
\begin{equation}\label{3.46}
\begin{aligned}
\Delta r_N(r_0) - \epsilon &\leq \dashint_{\partial B_p(r_0)} \Delta r \\&
= \frac{\int_{E_1}\Delta r}{A(\partial(B_p(r_0)))} + \frac{\int_{E_2}\Delta r}{A(\partial(B_p(r_0)))}
\\&\leq \frac{A(E_1)}{A(\partial(B_p(r_0)))}\Delta r_N(r_0) + (1 - \frac{A(E_1)}{A(\partial(B_p(r_0)))})(\Delta r_N(r_0) - \sqrt \epsilon\}).
\end{aligned}
\end{equation}
After a simple manipulation of (\ref{3.46}), claim \ref{cl6} follows.
\end{proof}

\begin{claim}\label{cl7}
{If $r_0 - 2 \leq r \leq r_0$, then $\frac{A(E_1{r})}{A(\partial B_p(r))} \geq 1 - \delta$.}
\end{claim}
\begin{proof}
Since we have two bounds for $\Delta r$,
\begin{equation}\label{3.47}
\begin{aligned}
&\frac{A(E_1(r))}{A(E_1(r_0))} \geq \frac{1}{C},\\&
\frac{A(E_2(r))}{A(E_2(r_0))} \leq C
\end{aligned}
\end{equation}
for $r_0 - 2 \leq r \leq r_0$.
Claim \ref{cl7} follows from (\ref{3.45}) and (\ref{3.47}).
\end{proof}

At the point $q = exp_p(r\theta)$, choose an orthonormal frame $\{d_1,...d_n\}$ near $q$ such that $d_1=\nabla r$; $Jd_{2\alpha-1}=d_{2\alpha}$ for $1\leq \alpha \leq m$. Define a unitary frame $\{e_{\alpha}\}$($\alpha = 1,...m$) so that $e_{\alpha} = \frac{1}{\sqrt{2}}(d_{2\alpha-1} - \sqrt{-1}d_{2\alpha})$ for all $\alpha$.
\begin{claim}\label{cl8}
{ Along the geodesic from $p$ satisfying $\theta \in F_1$, $\int\limits_{r_0-2}^{r_0}|r_{ij} - \overline r_{ij}|^2 \leq \delta$.}
\end{claim}
\begin{proof}
 We write the Bochner formula as
\begin{equation}\label{3.48}
\frac{\partial \Delta r}{\partial r} + \frac{1}{n-1}(\Delta r)^2 + \sum\limits_{i \neq j}r_{ij}^2 + \sum\limits_{i\neq 1}(\frac{\Delta r}{n-1} - r_{ii})^2 -n+1 \leq 0.
\end{equation}

According to (\ref{3.42}), we have
\begin{equation}\label{3.49}
\Delta r(r_0, \theta) \geq \Delta r_N(r_0) - \delta
\end{equation}
for $\theta \in F_1$.
After a simple analysis of (\ref{3.-13}), it follows that
\begin{equation}\label{3.50}
|\Delta r_N(r) - \Delta r(r, \theta)| < \delta
\end{equation}
for $r_0 - 2 \leq r \leq r_0$.
Integrating (\ref{3.48}) along the geodesic from $r_0 - 2$ to $r_0$,  in view of (\ref{3.50}), we find that
\begin{equation}\label{3.51}
\begin{aligned}
\int\limits_{r_0 - 2}^{r_0} (\sum\limits_{i \neq j}r_{ij}^2 + \sum\limits_{i\neq 1}(\frac{\Delta r}{n-1} - r_{ii})^2)
&\leq \int\limits_{r_0 - 2}^{r_0}(-\frac{\partial \Delta r}{\partial r} + n - 1 -\frac{1}{n-1}(\Delta r)^2 \\& =
 \Delta r(r_0-2, \theta) - \Delta r(r_0, \theta) +2(n-1) - \frac{1}{n-1}\int\limits_{r_0 - 2}^{r_0}(\Delta r)^2 \\&\leq \delta.
\end{aligned}
\end{equation}

Combining (\ref{3.50}) and (\ref{3.51}), claim \ref{cl8} follows.
\end{proof}

Applying (\ref{1.1}) to the distance function to $p$, we find
\begin{equation}\label{3.52}
\begin{aligned}
\frac{1}{2}\langle \nabla r, \nabla (\sum\limits_{\gamma \neq 1}r_{\gamma \overline{\gamma}})\rangle = r_{1\overline{1}}\Delta r -|r_{\alpha \overline{\beta}}|^2+ Re (div Y).
 \end{aligned}
\end{equation}
where $Y =\sum\limits_{\gamma \neq 1}r_{\overline{\alpha}}r_{\alpha \overline{\gamma}}e_{\gamma}$.
It is simple to see $$div|_M Y = div|_{\partial B_p(r)}Y.$$

Thus after the integration of (\ref{3.52}) on the geodesic sphere $\partial B_p(r)$, we find
 \begin{equation}\label{3.53}
 \int_{\partial B_p(r)} \frac{\partial \sum\limits_{\alpha \neq 1} r_{\alpha\overline{\alpha}}}{\partial r} = \int_{\partial B_p(r)} -2\sum\limits_{\alpha,\beta}|r_{\alpha\overline{\beta}}|^2+2r_{\alpha\overline{\alpha}}r_{1\overline{1}}.
 \end{equation}

For notational simplicity, we use $\int$ to denote $\int_{\partial B_p(r)}$, $\dashint$ to denote the average of $\int_{\partial B_p(r)}$.
Taking the average of (\ref{3.53}) on $\partial B_p(r)$, we get
  \begin{equation}\label{3.54}
 \frac{\partial \dashint \sum\limits_{\alpha \neq 1} r_{\alpha\overline{\alpha}}}{\partial r} = -2\dashint \sum\limits_{\alpha,\beta}|r_{\alpha\overline{\beta}}|^2+2\dashint r_{1\overline{1}}\Delta r  + 2\dashint \sum\limits_{\alpha\neq1}r_{\alpha\overline{\alpha}}\Delta r - 2\dashint \sum\limits_{\alpha\neq1} r_{\alpha\overline{\alpha}} \dashint \Delta r.
  \end{equation}

Integrating (\ref{3.54}) from $r-1$ to $r$, we find
\begin{equation}\label{3.55}
\begin{aligned}
\\&\dashint_{\partial B_p(r)}\sum\limits_{\alpha \neq 1} r_{\alpha\overline{\alpha}}dA -
\dashint_{\partial B_P(r-1)}\sum\limits_{\alpha \neq 1} r_{\alpha\overline{\alpha}}dA\\& = \int\limits_{r-1}^{r}
(-2\dashint \sum\limits_{\alpha,\beta}|r_{\alpha\overline{\beta}}|^2dA_t+2\dashint r_{1\overline{1}}\Delta rdA_t
\\& + 2\dashint \sum\limits_{\alpha\neq1}r_{\alpha\overline{\alpha}}\Delta rdA_t - 2\dashint \sum\limits_{\alpha\neq1}r_{\alpha\overline{\alpha}}dA_t \dashint \Delta rdA_t)dt.
\end{aligned}
\end{equation}

Integration of (\ref{3.55}) from $r_0 - 1$ to $r_0$ yields
\begin{equation}\label{3.56}
\begin{aligned}
&\int_{r_0-1}^{r_0}\dashint_{\partial B_p(r)}\sum\limits_{\alpha \neq 1} r_{\alpha\overline{\alpha}}dAdr -
\int_{r_0-1}^{r_0}\dashint_{\partial B_p(r-1)}\sum\limits_{\alpha \neq 1} r_{\alpha\overline{\alpha}}dAdr \\&= \int_{r_0-1}^{r_0}\int\limits_{r-1}^{r}
(-2\dashint \sum\limits_{\alpha,\beta}|r_{\alpha\overline{\beta}}|^2dA_t+2\dashint r_{1\overline{1}}\Delta rdA_t\\&
 + 2\dashint \sum\limits_{\alpha\neq1}r_{\alpha\overline{\alpha}}\Delta rdA_t - 2\dashint \sum\limits_{\alpha\neq1}r_{\alpha\overline{\alpha}}dA_t \dashint \Delta rdA_t) dtdr.
\end{aligned}
\end{equation}

We come to the estimate of the first term in the LHS of (\ref{3.56}).
\begin{equation}\label{3.57}
\begin{aligned}
&\int_{r_0-1}^{r_0}\dashint_{\partial B_p(r)}\sum\limits_{\alpha \neq 1} r_{\alpha\overline{\alpha}}dAdr\\& =
\int_{r_0-1}^{r_0}\frac{\int_{E_1(r)}\sum\limits_{\alpha \neq 1}(r_{\alpha\overline{\alpha}} - \overline{r}_{\alpha\overline{\alpha}})dA}{A(\partial B_p(r))} dr+ \int_{r_0-1}^{r_0} \frac{\int_{E_2(r)}\sum\limits_{\alpha \neq 1}r_{\alpha\overline{\alpha}}dA}{A(\partial B_p(r))}dr\\& + \int_{r_0-1}^{r_0}\frac{\int_{E_1(r)}\sum\limits_{\alpha \neq 1}\overline {r}_{\alpha\overline{\alpha}}dA}{A(\partial B_p(r))}dr.
\end{aligned}
\end{equation}

\begin{claim}\label{cl9}
\begin{equation}\label{3.58}
\begin{aligned}
&|\int_{r_0-1}^{r_0}\frac{\int_{E_1(r)}\sum\limits_{\alpha \neq 1}(r_{\alpha\overline{\alpha}} - \overline{r}_{\alpha\overline{\alpha}})dA}{A(\partial B_p(r))} dr| \leq \delta, \\&
|\int_{r_0-1}^{r_0}\frac{\int_{E_2(r)}\sum\limits_{\alpha \neq 1}r_{\alpha\overline{\alpha}}dA}{A(\partial B_p(r))}dr|
\leq \delta,\\&
|\int_{r_0-1}^{r_0}\frac{\int_{E_1(r)}\sum\limits_{\alpha \neq 1}\overline {r}_{\alpha\overline{\alpha}}dA}{A(\partial B_p(r))}dr - \int_{r_0-1}^{r_0}\overline {r}_{\alpha\overline{\alpha}}dr| \leq \delta.
\end{aligned}
\end{equation}
\end{claim}
\begin{proof}
By claim \ref{cl4}, for $r_0 - 2 \leq r \leq r_0$, we have the relation
\begin{equation}\label{3.59}
\frac{1}{C} \leq \frac{dA(\partial B_p(r))(\theta)}{dA(\partial B_p(r_0))(\theta)} \leq C
\end{equation}
where $dA(\partial B_p(r))(\theta)$ is the area element of $\partial B_p(r)$.
Therefore we get
\begin{equation}\label{3.60}
\begin{aligned}
\int_{r_0-1}^{r_0}\frac{\int_{E_1(r)}\sum\limits_{\alpha \neq 1}(r_{\alpha\overline{\alpha}} - \overline{r}_{\alpha\overline{\alpha}})dA}{A(\partial B_p(r))dA}dr &\leq \int_{r_0-1}^{r_0}\frac{\int_{E_1(r)}\sum\limits_{\alpha \neq 1}|r_{\alpha\overline{\alpha}} - \overline{r}_{\alpha\overline{\alpha}}|dA}{A(\partial B_p(r))}dr \\& \leq C\frac{\int_{r_0-1}^{r_0}\int_{E_1(r)}|r_{\alpha\overline{\alpha}} - \overline{r}_{\alpha\overline{\alpha}}|dA}
{A(\partial B_p(r_0))}dr\\& \leq C \frac{\int_{F_1}\int_{r_0-1}^{r_0}|r_{\alpha\overline{\alpha}} - \overline{r}_{\alpha\overline{\alpha}}|drdA_{r_0}}{A(\partial B_p(r_0))} \\& \leq C \frac{\int_{F_1(r)}(\int_{r_0-1}^{r_0}\sum\limits_{\alpha \neq 1}|r_{\alpha\overline{\alpha}}-\overline r_{\alpha\overline{\alpha}}|^2dr)^{\frac{1}{2}}dA_{r_0}}{A(\partial B_p(r_0))} \\&\leq \delta.
\end{aligned}
\end{equation}

In the last inequality, we have used claim \ref{cl6} and claim \ref{cl8}. Similarly
\begin{equation}\label{3.61}
\begin{aligned}
\int_{r_0-1}^{r_0}\frac{\int_{E_2(r)}\sum\limits_{\alpha \neq 1}r_{\alpha\overline{\alpha}}dA}{A(\partial B_p(r))}dr &\leq \int_{r_0-1}^{r_0}\frac{\int_{E_2(r)}\sum\limits_{\alpha \neq 1}|r_{\alpha\overline{\alpha}}|dA}{A(\partial B_p( r))}dr\\&
\leq C\frac{\int_{r_0-1}^{r_0}\int_{E_2(r)}\sum\limits_{\alpha \neq 1}|r_{\alpha\overline{\alpha}}|dA}{A(\partial B_p( r_0))}dr\\& \leq C \frac{\int_{F_2(r)}\int_{r_0-1}^{r_0}\sum\limits_{\alpha \neq 1}|r_{\alpha\overline{\alpha}}|drdA_{r_0}}{A(\partial B_p(r_0))} \\& \leq C \frac{\int_{F_2(r)}(\int_{r_0-1}^{r_0}\sum\limits_{\alpha \neq 1}|r_{\alpha\overline{\alpha}}|^2dr)^{\frac{1}{2}}dA_{r_0}}{A(\partial B_p(r_0))} \\&\leq \delta.
\end{aligned}
\end{equation}

\begin{equation}\label{3.62}
\begin{aligned}
|\int_{r_0-1}^{r_0}\frac{\int_{E_1(r)}\sum\limits_{\alpha \neq 1}\overline {r}_{\alpha\overline{\alpha}}dA}{A(\partial B_p(r))}dr - \int_{r_0-1}^{r_0}\sum\limits_{\alpha \neq 1}\overline {r}_{\alpha\overline{\alpha}}dr|& =
\int_{r_0-1}^{r_0} (1 - \frac{A(E_1(r))}{A(\partial B_p(r))})\sum\limits_{\alpha \neq 1}\overline {r}_{\alpha\overline{\alpha}}dr \\& \leq \delta.
\end{aligned}
\end{equation}

This completes the proof of claim \ref{cl9}.
\end{proof}

Claim \ref{cl9} says up to a negligible error, we can replace $r_{\alpha\overline\beta}$ by $\overline r_{\alpha\overline\beta}$ in the first term of (\ref{3.56}) , where $\overline r_{\alpha\overline\beta}$ is the complexification of $\overline r_{ij}$. Similarly, we can apply the replacement to all other terms in (\ref{3.56}) with a negligible error. In order to get a contradiction to (\ref{3.41}), we just need to prove that after the replacement, there is an explicit gap between the LHS and the RHS of (\ref{3.56}). It suffices to find a gap between of LHS and RHS of (\ref{1.1}) after the replacement $f_{\alpha\overline\beta} \rightarrow \overline r_{\alpha\overline\beta}, \nabla f \rightarrow \nabla r$. The computation of the gap is the same as (\ref{3.-12}).

  We have thus proved theorem \ref{thm2} when $p$ is a pole, $Ric \geq -(2m-1)$, $r > 5$. The proof of the general case is similar. \qed

\section{\bf{ Average Laplacian comparison for some special cases}}
\noindent\emph{Proof of theorem \ref{thm3}:}

For simplicity, we write $r_M(x)$ as $r$. Near a point $q \in M$, choose a unitary frame $\{e_{\alpha}\}\in T^{1,0}(M)$ such that $e_1 = \frac{1}{\sqrt{2}}(\nabla r - \sqrt{-1}J\nabla r)$. Since the metric is unitary invariant, it is simple to see $r_{\alpha\overline{\beta}}=0$ if $\alpha \neq \beta$; $r_{\alpha\overline{\alpha}} = r_{\beta\overline{\beta}}$ if $\alpha \neq 1$ and $\beta \neq 1$; $r_{\alpha\beta}=0$ unless $\alpha = \beta =1$; $r_{11} = -r_{1\overline 1}$.

Using $r_{1\overline{1}} = \Delta r - (m-1)r_{2\overline{2}}$ and (\ref{2.1}), we find
\begin{equation}\label{4.1}
\frac{1}{2}R_{1\overline{1}} + \frac{\partial \Delta r}{\partial r} + (m-1)|r_{2\overline{2}}|^2 + 2(\Delta r - (m-1)r_{2\overline{2}})^2=0.
\end{equation}

Applying (\ref{1.1}) to $r$, after a simplification, we have
\begin{equation}\label{4.2}
\frac{\partial r_{2\overline{2}}}{\partial r} = 2r_{2\overline{2}}(\Delta r - m r_{2\overline{2}}).
\end{equation}

We use $\overline\Delta r$ and $\overline r_{\alpha\overline\beta}$ to denote the Laplacian and complex Hessian of the distance function in $M_k$. Let us write down the equations for $M_k$ analogue to (\ref{4.1}) and (\ref{4.2}). Explicitly,
\begin{equation}\label{4.3}
\frac{1}{2}(m+1)k + \frac{\partial \overline{\Delta} r}{\partial r} + (m-1)|\overline{r}_{2\overline{2}}|^2 + 2(\overline{\Delta} r - (m-1)\overline{r}_{2\overline{2}})^2=0,
\end{equation}
\begin{equation}\label{4.4}
\frac{\partial \overline{r}_{2\overline{2}}}{\partial r} = 2\overline{r}_{2\overline{2}}(\overline{\Delta} r - m \overline{r}_{2\overline{2}}).
\end{equation}

We shall regard $r_{\alpha\overline{\beta}}$ and $\overline r_{\alpha\overline{\beta}}$ as functions of $r$.
(\ref{4.1}), (\ref{4.3}) give
\begin{equation}\label{4.5}
\begin{aligned}
(\overline{\Delta} r - \Delta r)' &\geq \frac{1}{2}(R_{1\overline{1}} - (m+1)k) -f(r)| \overline{\Delta} r - \Delta r| -
 f(r) |\overline{r}_{2\overline{2}} - r_{2\overline{2}}| \\&\geq -f(r)| \overline{\Delta} r - \Delta r| -
 f(r) |\overline{r}_{2\overline{2}} - r_{2\overline{2}}|.
 \end{aligned}
 \end{equation}

Similarly (\ref{4.2}), (\ref{4.4}) give
\begin{equation}\label{4.6}
 (\overline{r}_{2\overline{2}} - r_{2\overline{2}})' \geq -f(r)| \overline{\Delta} r - \Delta r| -
 f(r) |\overline{r}_{2\overline{2}} - r_{2\overline{2}}|.
 \end{equation}

In (\ref{4.5}) and (\ref{4.6}), $f(r)$ is a suitable positive function depending only on the metric $g$ of $M$.
(\ref{4.5}) and (\ref{4.6}) yield
\begin{equation}\label{4.7}
(\overline{\Delta} r - \Delta r+(\overline{r}_{2\overline{2}} - r_{2\overline{2}}))' \geq -2f(r)(| \overline{\Delta} r - \Delta r| +
 |\overline{r}_{2\overline{2}} - r_{2\overline{2}}|).
 \end{equation}

We divide the proof of theorem \ref{thm3} into two cases. Namely, $k = -1$ and $k = 1$. Note that the case $k = 0$ is included in the real Laplacian comparison theorem.

First consider the case $k = -1$. It suffices to prove the claim below:
\begin{claim}\label{cl10}
{$\overline{\Delta} r - \Delta r$ and $\overline{r}_{2\overline{2}} - r_{2\overline{2}}$ are always nonnegative.}
\end{claim}
\begin{proof}
First we prove the claim under the assumption as follows:
\begin{equation}\label{4.-1}
\Delta r < \overline\Delta r \ \text{and} \ r_{2\overline{2}} < \overline r_{2\overline{2}}\ \text{when $r$ is small}.
\end{equation}
If the claim is not true, there are three possibilities.

1.  When $r$ is increasing, $\overline{\Delta} r - \Delta r$ is becoming negative before  $\overline{r}_{2\overline{2}} - r_{2\overline{2}}$ does.

2. $\overline r_{2\overline{2}} - r_{2\overline{2}}$  is becoming negative before $\overline{\Delta} r - \Delta r$ does.

3. There exists a constant $r_0 > 0$ such that $\overline r_{2\overline{2}}|_{r=r_0} = r_{2\overline{2}}|_{r=r_0}$,  $\overline{\Delta} r|_{r=r_0} = \Delta r|_{r=r_0}$. $\Delta r < \overline\Delta r$ and $r_{2\overline{2}} < \overline r_{2\overline{2}}$ for $r < r_0$.

\bigskip

For case 1, let $r = r_0>0$ be the first radius such that $\overline{\Delta} r - \Delta r$ is becoming negative while $\overline{r}_{2\overline{2}}(r_0) - r_{2\overline{2}}(r_0) > 0$.

 We are going to prove
\begin{equation}\label{4.8}
 (\overline{\Delta} r - \Delta r)'|_{r=r_0} > 0.
  \end{equation}

(\ref{4.1}) and (\ref{4.3}) give
\begin{equation}\label{4.9}
\begin{aligned}
(\overline{\Delta} r - \Delta r)' &\geq 2(\Delta r - (m-1)r_{2\overline{2}})^2 + (m-1)r_{2\overline{2}}^2
\\&-(2(\overline\Delta r - (m-1)\overline{r}_{2\overline{2}})^2 + (m-1)\overline{r}_{2\overline{2}}^2).
\end{aligned}
\end{equation}

To prove (\ref{4.8}), it suffices to prove
\begin{equation}\label{4.10}
2(\Delta r - (m-1)r_{2\overline{2}})^2 + (m-1)r_{2\overline{2}}^2 - 2(\overline\Delta r - (m-1)\overline{r}_{2\overline{2}})^2 - (m-1)\overline{r}_{2\overline{2}}^2 > 0.
\end{equation}

Since $(\Delta r - \overline\Delta r)|_{r=r_0} = 0$, after a simplification, (\ref{4.10}) is equivalent to
\begin{equation}\label{4.11}
(\overline{r}_{2\overline{2}} - r_{2\overline{2}})(4\Delta r -(2m-1)(r_{2\overline{2}}+\overline{r}_{2\overline{2}})) >0.
\end{equation}

According to the assumption in case 1, we have
\begin{equation}\label{4.12}
(\overline{r}_{2\overline{2}} - r_{2\overline{2}})|_{r=r_0} > 0.
\end{equation}

Using $k = -1$ and the assumption in case 1, we find
\begin{equation}\label{4.13}
\Delta r|_{r=r_0}=\overline\Delta r|_{r=r_0} > \frac{2m-1}{2}\overline r_{\overline{2}2}|_{r=r_0}.
\end{equation}

Therefore
\begin{equation}\label{4.14}
(4\Delta r -(2m-1)(r_{2\overline{2}}+\overline{r}_{2\overline{2}}))|_{r=r_0} > 0.
\end{equation}

Putting (\ref{4.12}) and (\ref{4.14}) together, we obtain the proof of (\ref{4.11}) and (\ref{4.8}). However, (\ref{4.8}) contradicts the assumption that
$\overline{\Delta} r - \Delta r$ is becoming negative when $r= r_0$.
Therefore case 1 can not happen.

\bigskip

Now consider case 2. Let $r = r_0>0$ be the first radius such that $\overline{r}_{2\overline{2}}-r_{2\overline{2}}$ is becoming negative while $\overline{\Delta} r(r_0) - \Delta r(r_0) > 0$.

Using (\ref{4.2}), (\ref{4.4}), $r_{2\overline{2}}|_{r=r_0} =\overline r_{2\overline{2}}|_{r=r_0}$ and $\overline r_{2\overline 2} > 0$, we find
\begin{equation}\label{4.15}
(\overline{r}_{2\overline{2}} - r_{2\overline{2}})'|_{r=r_0}= (2\overline r_{2\overline 2}(\overline\Delta r - \Delta r))|_{r=r_0} >0
\end{equation}

(\ref{4.15}) contradicts the assumption that $\overline{r}_{2\overline{2}}-r_{2\overline{2}}$ is becoming negative
at $r = r_0$. Therefore case 2 can not happen.

\bigskip

Consider case 3 now.
Using the assumption that $\Delta r < \overline\Delta r$ and $r_{2\overline{2}} < \overline r_{2\overline{2}}$ for $r < r_0$, we integrate (\ref{4.7}) from $\frac{r_0}{2}$ to $r_0$. It follows
\begin{equation}\label{4.16}
(\overline{\Delta} r - \Delta r+ \overline{r}_{2\overline{2}} - r_{2\overline{2}})|_{r=r_0} >0.
 \end{equation}

This contradicts the assumption of case 3.

\bigskip

So far we have proved claim \ref{cl10} under the condition (\ref{4.-1}). For general case, let
$\tilde g = (1+\epsilon)g$ where $\epsilon$ is a small positive constant. It is simple to check that $\tilde g$ satisfies $Ric (\tilde g) \geq -(m+1)\tilde g$. After a simple computation of the asymptotic expansion of $\Delta_{\tilde g}r$ and $(r_{\tilde g})_{2\overline 2}$ for small $r$, $\tilde g$ satisfies (\ref{4.-1}).
The proof of claim \ref{cl10} is complete if we let $\epsilon$ approach 0.

\end{proof}
The proof of the case $k = -1$ is complete.

\bigskip

Now we turn to the case $k = 1$.

\begin{claim}\label{cl11}
{For any $r_0>0$, if
\begin{equation}\label{4.17}
\overline{r}_{2\overline{2}} - r_{2\overline{2}} \geq 0, \overline\Delta r - \Delta r - (m-1)(\overline{r}_{2\overline{2}} - r_{2\overline{2}})\geq 0
\end{equation}
for any $r \in (0, r_0)$, there exists a function $g(r)$ such that
\begin{equation}\label{4.18}
(\overline\Delta r - \Delta r - (m-1)(\overline{r}_{2\overline{2}} - r_{2\overline{2}}))' \geq g(r)(\overline\Delta r - \Delta r - (m-1)(\overline{r}_{2\overline{2}} - r_{2\overline{2}}))
\end{equation} for $r \in (0, r_0)$.}
\end{claim}
\begin{proof}

By (\ref{4.1}), (\ref{4.2}), (\ref{4.3}) and (\ref{4.4}), we have
 \begin{equation}\label{4.19}
\begin{aligned}
&(\overline\Delta r - \Delta r - (m-1)(\overline{r}_{2\overline{2}} - r_{2\overline{2}}))' \\& =
\frac{1}{2}(R_{1\overline 1}- (m+1))+ (m-1)r_{2\overline 2}^2 + 2(\Delta r - (m-1)r_{2\overline 2})^2
- (m-1)\overline r_{2\overline 2}^2 \\&- 2(\overline\Delta r - (m-1)\overline r_{2\overline 2})^2
+(m-1)(2r_{2\overline 2}\Delta r - 2mr_{2\overline 2}^2 - 2\overline r_{2\overline 2}\overline\Delta r + 2m \overline r_{2\overline 2}^2) \\& \geq (m-1+2(m-1)^2-2m(m-1))r_{2\overline 2}^2 + 2(\Delta r)^2 + (-4(m-1) + 2(m-1))r_{2\overline 2}\Delta r \\&-((m-1+2(m-1)^2-2m(m-1))\overline r_{2\overline 2}^2 + 2(\overline\Delta r)^2 + (-4(m-1) + 2(m-1))\overline r_{2\overline 2}\overline\Delta r) \\& = -(m-1)r_{2\overline 2}^2 + 2(\Delta r)^2 -2(m-1)r_{2\overline 2}\Delta r - (-(m-1)\overline r_{2\overline 2}^2 + 2(\overline\Delta r)^2 -2(m-1)\overline r_{2\overline 2}\overline\Delta r) \\& =(m-1)(\overline r_{2\overline 2} - r_{2\overline 2})(\overline r_{2\overline 2} + r_{2\overline 2}) - 2(\overline\Delta r - \Delta r)(\overline\Delta r + \Delta r) \\&+ 2(m-1)\overline\Delta r(\overline r_{2\overline 2} - r_{2\overline 2}) + 2(m-1)r_{2\overline 2}(\overline\Delta r - \Delta r) \\& =
(\overline\Delta r - \Delta r)(2(m-1)r_{2\overline 2} - 2(\overline\Delta r + \Delta r)) + (m-1)(\overline r_{2\overline 2} - r_{2\overline 2})(\overline r_{2\overline 2} + r_{2\overline 2}+2\overline\Delta r).
\end{aligned}
\end{equation}

Note that $\overline\Delta r - \Delta r \geq 0$ when $r < r_0$. To prove claim \ref{cl11}, by (\ref{4.18}) and (\ref{4.19}), it suffices to find a function $g(r)$ satisfying the two inequalities below:
\begin{equation}\label{4.20}
2(m-1)r_{2\overline{2}} - 2(\overline\Delta r + \Delta r)\geq g,
\end{equation}
\begin{equation}\label{4.21}
(m-1)(\overline r_{2\overline 2} + r_{2\overline 2}+2\overline\Delta r) \geq -(m-1)g.
\end{equation}

 We take
\begin{equation}\label{4.22}
g = - 2\overline\Delta r -(\overline{r}_{2\overline{2}} + r_{2\overline{2}}),
\end{equation}
then $g$ satisfies (\ref{4.21}).
Plugging $g$ in (\ref{4.20}), after a slight simplification, it suffices to prove
\begin{equation}\label{4.23}
2(m-1)r_{2\overline{2}} + \overline{r}_{2\overline{2}} + r_{2\overline{2}} \geq 2\Delta r.
\end{equation}

Since $k = 1$, a simple computation gives
\begin{equation}\label{4.24}
\overline\Delta{r} < \frac{2m-1}{2}\overline{r}_{2\overline{2}}.
\end{equation}

Putting (\ref{4.17}) and (\ref{4.24}) together, we get
\begin{equation}\label{4.25}
\Delta r - (m-1)r_{2\overline{2}} \leq \overline\Delta r-(m-1)\overline{r}_{2\overline{2}} \leq \frac{\overline{r}_{2\overline{2}}}{2}.
\end{equation}

  An observation of (\ref{4.2}) gives
 \begin{equation}\label{4.26}
  r_{2\overline{2}} > 0.
  \end{equation}

 (\ref{4.25}) and (\ref{4.26}) imply (\ref{4.23}).
 The proof of claim \ref{cl11} is complete.
\end{proof}

Since $\Delta r = (m-1)r_{2\overline 2} + (\Delta r - (m-1)r_{2\overline 2})$, to prove theorem \ref{thm3} for the case $k = 1$, it suffices to prove the claim as follows:
\begin{claim}\label{cl12}
{ $\overline r_{2\overline{2}} - r_{2\overline{2}} \geq 0$, $\overline\Delta r - \Delta r - (m-1)(\overline r_{2\overline{2}} - r_{2\overline{2}}) \geq 0$ for all $r$.}
\end{claim}
\begin{proof}
We first prove claim \ref{cl12} under the assumption that $\overline r_{2\overline{2}} - r_{2\overline{2}}$ and $\overline\Delta r - \Delta r - (m-1)(\overline r_{2\overline{2}} - r_{2\overline{2}})$ are positive for small $r$.
If claim \ref{cl12} does not hold, there are three possibilities.

1.  When $r$ is increasing, $\overline\Delta r - \Delta r - (m-1)(\overline r_{2\overline{2}} - r_{2\overline{2}})$ is becoming negative before  $\overline{r}_{2\overline{2}} - r_{2\overline{2}}$ does.

2. $\overline r_{2\overline{2}} - r_{2\overline{2}}$ is becoming negative before $\overline\Delta r - \Delta r - (m-1)(\overline r_{2\overline{2}} - r_{2\overline{2}})$ does.

3. There exists a constant $r_0 > 0$ such that $\overline r_{2\overline{2}}|_{r=r_0} = r_{2\overline{2}}|_{r=r_0}$,  $\overline\Delta r - \Delta r - (m-1)(\overline r_{2\overline{2}} - r_{2\overline{2}})|_{r=r_0} =0$. $\overline\Delta r - \Delta r - (m-1)(\overline r_{2\overline{2}} - r_{2\overline{2}}) > 0$ and $\overline r_{2\overline{2}} - r_{2\overline{2}} > 0$ for $r < r_0$.

\bigskip

For case 1, we apply claim \ref{cl11}. Let $r = r_0>0$ be the first radius such that $\overline\Delta r - \Delta r - (m-1)(\overline r_{2\overline{2}} - r_{2\overline{2}})$ is becoming negative while $\overline{r}_{2\overline{2}}(r_0) - r_{2\overline{2}}(r_0) > 0$. Integrating (\ref{4.18}) from $\frac{r_0}{2}$ to $r_0$, we find
$\overline\Delta r - \Delta r - (m-1)(\overline r_{2\overline{2}} - r_{2\overline{2}})|_{r=r_0} > 0$.
This contradicts the assumption of case 1.

\bigskip

For case 2 and case 3, we can get the same contradiction as in the proof of claim \ref{cl10}.

\bigskip

Thus we have completed the proof of claim \ref{cl12} under the condition that $\overline r_{2\overline{2}} - r_{2\overline{2}}$ and $\overline\Delta r - \Delta r - (m-1)(\overline r_{2\overline{2}} - r_{2\overline{2}})$ are positive for small $r$.
To remove the condition, we can use the same strategy as in the proof of claim \ref{cl10}.
 \end{proof}

This proves the case when $k = 1$.
The proof of theorem \ref{thm3} is complete.

\qed

\bigskip

Using $\dashint$ to denote the average on the geodesic sphere $\partial B_p(r)$, we are going to prove the following theorem:
\begin{theorem}\label{thm5}
{ Let $M^m$ be a complete K\"ahler manifold such that $Ric \geq -(m+1)$. Consider a point $p \in M$, define $r$ to be the distance function to $p$ on $M$. Near a point $q \in M$, choose a unitary frame $\{e_{\alpha}\}\in T^{1,0}(M)$ such that $e_1 = \frac{1}{\sqrt{2}}(\nabla r - \sqrt{-1}J\nabla r)$. Let $\overline M^m$ be the simply connected complex space form with constant bisectional curvature $-1$. Use $\overline r_{\alpha\overline{\beta}}$ and $\overline\Delta r$ to denote the complex hessian and Laplacian of the distance function on $\overline M$. If either $\Delta r$ or $\sum\limits_{\alpha \neq 1}r_{\alpha \overline{\alpha}}$ is a function of $r$, then inside the injective radius of $p$, the following average Laplacian comparison holds,
\begin{equation}\label{4.27}
\dashint \Delta r \leq \overline{\Delta}r(r).
\end{equation}}
\end{theorem}
\begin{proof}

To prove theorem \ref{thm5}, it suffices to prove the following claim:
\begin{claim}\label{cl13}
{\begin{equation}\label{4.33}
\dashint \Delta r \leq \overline\Delta r, \dashint \sum\limits_{\alpha \neq 1}r_{\alpha\overline\alpha} \leq
\sum\limits_{\alpha\neq 1}\overline r_{\alpha\overline\alpha}.
\end{equation}}
\end{claim}
\begin{proof}

Let $u(r) = \dashint\Delta r, v(r) = \dashint \sum\limits_{\alpha \neq 1}r_{\alpha\overline\alpha}$.
Integrating the Bochner formula (\ref{2.1}) on the geodesic sphere, we find
 $$\int_{\partial B_p(r)} \frac{\partial \Delta r}{\partial r}  \leq \int_{\partial B(P,r)}-\frac{1}{2}Ric_{1\overline{1}}-2(\Delta r - \sum\limits_{\alpha \neq 1}r_{\alpha\overline{\alpha}})^2-\sum\limits_{\alpha \neq 1}
 r^2_{\alpha\overline{\alpha}}.$$
 Taking the average on the geodesic sphere $\partial B_p(r)$, we get
 \begin{equation}\label{4.28}
 \frac{\partial \dashint \Delta r}{\partial r} \leq -\frac{1}{2}\dashint Ric_{1\overline{1}} + 2\dashint (\Delta r)^2 - 2(\dashint \Delta r)^2 -
2\dashint (\Delta r - \sum\limits_{\alpha\neq1}r_{\alpha\overline{\alpha}})^2 - \sum\limits_{\alpha\neq1}\dashint |r_{\alpha\overline{\alpha}}|^2.
\end{equation}

If $\Delta r$ is a function of $r$, then (\ref{4.28}) becomes
  \begin{equation}\label{4.29}
  \begin{aligned}
 \frac{\partial \dashint \Delta r}{\partial r} &\leq -\frac{1}{2}\dashint Ric_{1\overline{1}} -
2\dashint (\Delta r - \sum\limits_{\alpha\neq1}r_{\alpha\overline{\alpha}})^2 - \sum\limits_{\alpha\neq1}\dashint |r_{\alpha\overline{\alpha}}|^2.
\end{aligned}
\end{equation}
After a slight simplification, we obtain
\begin{equation}\label{4.-2}
  u' \leq -\frac{1}{2}(\dashint R_{1\overline 1}) - 2u^2 + 4uv - \frac{2m-1}{m-1}v^2.
\end{equation}

(\ref{3.54}) becomes
\begin{equation}\label{4.30}
\begin{aligned}
 \frac{\partial \dashint \sum\limits_{\alpha \neq 1} r_{\alpha\overline{\alpha}}}{\partial r} = -2\dashint \sum\limits_{\alpha,\beta}|r_{\alpha\overline{\beta}}|^2+2\dashint r_{1\overline{1}}\Delta r.
 \end{aligned}
  \end{equation}
  Further simplification gives
 \begin{equation}\label{4.-3}
\begin{aligned}
v' \leq 2uv - \frac{2m}{m-1}v^2.
\end{aligned}
  \end{equation}

\bigskip
  If $\sum\limits_{\alpha \neq 1}r_{\alpha \overline{\alpha}}$ is a function of $r$, then (\ref{4.28}) becomes
 \begin{equation}\label{4.31}
 \begin{aligned}
 \frac{\partial \dashint \Delta r}{\partial r} \leq -\frac{1}{2}\dashint Ric_{1\overline{1}} - 2(\dashint \Delta r)^2-2\dashint |\sum\limits_{\alpha\neq1}r_{\alpha\overline{\alpha}}|^2 - \sum\limits_{\alpha\neq1}\dashint |r_{\alpha\overline{\alpha}}|^2 +4\dashint (\Delta r\sum\limits_{\alpha\neq1}r_{\alpha\overline{\alpha}})
 \end{aligned}
 \end{equation}
 where we expanded the term $2\dashint (\Delta r - \sum\limits_{\alpha\neq1}r_{\alpha\overline{\alpha}})^2$ in (\ref{4.28}).
 (\ref{4.31}) is equivalent to (\ref{4.-2}).

 For (\ref{3.54}), we write it as
\begin{equation}\label{4.32}
\begin{aligned}
\frac{\partial \dashint \sum\limits_{\alpha \neq 1} r_{\alpha\overline{\alpha}}}{\partial r} = -2\dashint \sum\limits_{\alpha\neq 1,\beta \neq1}|r_{\alpha\overline{\beta}}|^2+ 2\dashint (\Delta r -\sum\limits_{\alpha\neq1}r_{\alpha\overline{\alpha}})\sum\limits_{\alpha\neq1}r_{\alpha\overline{\alpha}}
\end{aligned}
\end{equation}
which could be written as (\ref{4.-3}).

Combining (\ref{4.-2}), (\ref{4.-3}), the proof of claim \ref{cl13} is almost the same as the proof of claim \ref{cl10}, so we skip the proof here.
\end{proof}
We complete the proof of theorem \ref{thm5}.
\end{proof}
\begin{remark} {Note that in the proof of theorem \ref{thm5}, one just needs to assume $\dashint Ric_{1\overline{1}}$ to be bounded from below by a negative constant.}
\end{remark}

\section{\bf{ An example}}

 In this section, we give a simple example showing that when the Ricci curvature is bounded from below by a positive constant, the diameter of the K\"ahler manifold could exceed the diameter of the complex space forms. This implies
 that in general situation, the sharp version of theorem \ref{thm1} is not true comparing with the complex space forms.

Let $N^m = \mathbb{C}\mathbb{P}^1\times\cdot\cdot\cdot\times\mathbb{C}\mathbb{P}^1$ to be the K\"ahler manifold equipped with the product metric, each $\mathbb{C}\mathbb{P}^1$ has the Fubini-Study metric. We can rescale $N^m$ so that $Ric = g$. It is simple to see $$diam(N^m)=\sqrt{m}\pi.$$

After a rescaling, $\mathbb{C}\mathbb{P}^m$ inherits a K\"ahler-Einsten metric with $Ric = g$. Given a unit vector $X \in T(\mathbb{C}\mathbb{P}^m)$, one can see that $$R_{XJXJXX} = \frac{2}{m+1},$$ therefore  $$diam(\mathbb{C}\mathbb{P}^m)=\frac{\pi}{\sqrt{\frac{2}{m+1}}}.$$

If $m > 1$, one sees that
$$diam(N^m) > diam(\mathbb{C}\mathbb{P}^m).$$

One can compare this example with the result of Li and Wang in \cite{[LW]}. Their theorem says that for a complete K\"ahler manifold, if the bisectional curvature is bounded from below by a positive constant, then $\mathbb{C}\mathbb{P}^m$ has the maximal diameter. We also compare the example with the result in \cite{[L]} by the author.
\begin{theorem}\label{thm7}
{Let $M^m$ be a K\"{a}hler manifold with real analytic metric. Suppose $Ric \geq K$ ($K$ is any real number), then given a point $p\in M$, for sufficiently small $r>0$, the area of geodesic spheres satisfies $A(\partial B_p(r))\leq A(\partial B_{N_K}(r))$,  where $N_K$ denotes the rescaled complex space form with $Ric = K$. The equality holds iff the $M$ is locally isometric to $N_K$.}
\end{theorem}

 If we apply theorem \ref{thm7} to the example, then for small $r$, $$A(\partial B_{N^m}(r))\leq A(\partial B_{\mathbb{C}\mathbb{P}^m}(r)).$$

 However, if $r$ lies between $diam(\mathbb{C}\mathbb{P}^m)$
 and $diam(N^m)$, then the inequality does not hold. It is not clear to the author whether the sharp version of theorem \ref{thm1} is true when the Ricci curvature is bounded from below by a negative constant. We can show that along the diagonal of $\mathbb{C}\mathbb{P}^1\times\mathbb{C}\mathbb{P}^1$,  the Laplacian of the distance function is greater than that of $\mathbb{C}\mathbb{P}^2$. However, the Laplacian of the distance function in $\mathbb{C}\mathbb{H}^1\times \mathbb{C}\mathbb{H}^1$ along the diagonal is smaller than that of $\mathbb{C}\mathbb{H}^2$.

\section{\bf{ Gradient estimate}}

\noindent\emph{Proof of theorem \ref{thm4}:}

Let us recall the following theorem due to Yau \cite{[Y]}:
\begin{theorem}\label{thm8}
{Let $M^n$ be a complete Riemannian manifold with Ricci curvature bounded from below by -(n-1). If f is a positive harmonic function on $M$, then
\begin{equation}\label{6.1}
|\nabla \log f| \leq n-1.
\end{equation}}
\end{theorem}
Set $n = 2m$, $h = \log f$. By direct computation, we find
\begin{equation}\label{6.2}
\Delta h = -|\nabla h|^2.
\end{equation}

 At a point $p \in M$ such that $\nabla h \neq 0$, choose an orthonormal frame $\{d_1,...d_n\}$ near $p$ such that $d_1=\frac{\nabla h}{|\nabla h|}$; $Jd_{2\alpha-1}=d_{2\alpha}$ for $1\leq \alpha \leq m$. Define a unitary frame $\{e_{\alpha}\}$($\alpha = 1,...m$) so that $e_{\alpha} = \frac{1}{\sqrt{2}}(d_{2\alpha-1} - \sqrt{-1}d_{2\alpha})$ for all $\alpha$.

 Using the Bochner formula, we compute
  \begin{equation}\label{6.3}
\begin{aligned}
\Delta |\nabla h|^2 &= 2h^2_{ij} + 2Ric(\nabla h, \nabla h) + 2\langle \nabla h, \nabla \Delta h\rangle
\\&\geq 2\sum\limits_{i\neq j}h^2_{ij}+2h_{11}^2 + \frac{2(\Delta h - h_{11})^2}{n-1} + 2\sum\limits_{i\neq1}(\frac{\Delta h - h_{11}}{n-1}-h_{ii})^2\\&-2(n-1)|\nabla h|^2 - 2
\langle \nabla h, \nabla |\nabla h|^2\rangle\\& = 2\sum\limits_{i\neq j}h^2_{ij}+2h_{11}^2 + 2\sum\limits_{i\neq1}(\frac{\Delta h - h_{11}}{n-1}-h_{ii})^2\\&+\frac{2}{n-1}(|\nabla h|^4 + 2|\nabla h|^2h_{11}+h_{11}^2)-2(n-1)|\nabla h|^2 - 2
\langle \nabla h, \nabla |\nabla h|^2\rangle \\&= 2\sum\limits_{i\neq j}h^2_{ij}+\frac{2n}{n-1}h_{11}^2
+ 2\sum\limits_{i\neq1}(\frac{\Delta h- h_{11}}{n-1}-h_{ii})^2\\&+\frac{2}{n-1}|\nabla h|^4-2(n-1)|\nabla h|^2
-\frac{2n-4}{n-1}\langle \nabla h, \nabla |\nabla h|^2\rangle.
\end{aligned}
\end{equation}

In the computation above, we have used the fact
\begin{equation}\label{6.4}
\langle \nabla h, \nabla |\nabla h|^2\rangle = h_i(h_j^2)_i=2|\nabla h|^2h_{11}.
\end{equation}

Now we define
\begin{equation}\label{6.5}
\begin{aligned}
u = 2\sum\limits_{i\neq j}h^2_{ij}+\frac{2n}{n-1}&h_{11}^2
+ 2\sum\limits_{i\neq1}(\frac{\Delta h- h_{11}}{n-1}-h_{ii})^2 \geq 0,\\&
g = |\nabla h|^2.
\end{aligned}
\end{equation}

Theorem \ref{thm8} says that
\begin{equation}\label{6.6}
0 \leq g \leq (n-1)^2.
\end{equation}

We may write (\ref{6.3}) as
\begin{equation}\label{6.7}
\begin{aligned}
\Delta g &\geq u + \frac{2}{n-1}g^2-2(n-1)g - \frac{2n-4}{n-1}\langle \nabla h, \nabla g\rangle
\\& = u + \frac{2}{n-1}g(g-(n-1)^2) -\frac{2n-4}{n-1}\langle \nabla h, \nabla g\rangle \\& \geq u + 2(n-1)(g-(n-1)^2)- \frac{2n-4}{n-1}\langle \nabla h, \nabla g\rangle.
\end{aligned}
\end{equation}
In the second inequality we have used (\ref{6.6}).
Define a new function
\begin{equation}\label{6.8}
w = (n-1)^2 - g,
\end{equation}
then
\begin{equation}\label{6.9}
0 \leq w \leq (n-1)^2.
\end{equation}
Moreover, $w$ satisfies the inequality
$$-\Delta w \geq u - 2(n-1)w + \frac{2n-4}{n-1}\langle \nabla h, \nabla w\rangle,$$
that is,
\begin{equation}\label{6.10}
\Delta w + \frac{2n-4}{n-1}\langle \nabla h, \nabla w\rangle +u \leq 2(n-1)w.
\end{equation}

Let us invoke a theorem in \cite{[LI]}, page 76, which is proved by the standard Di Georgi-Nash-Moser iteration:
\begin{theorem}\label{thm9}
{ Let $M^n$ be a complete Riemannian manifold with $Ric \geq k$. Let $p$ be a point in $M$. If $f$ is a nonnegative function on $M$ satisfying the inequality $$\Delta f \leq Af$$ for some constant $A \geq 0$, then there exist positive constants $\lambda, C$ depending only on $r, A, k, n$ such that $$(\dashint_{B_p(r)}f^{\lambda})^{\frac{1}{\lambda}} \leq C \inf\limits_{B_p{(\frac{r}{16})}}f.$$}
\end{theorem}

We would like to apply theorem \ref{thm9} to the function $w$ in (\ref{6.10}). The situation is a little bit different: there is a first order term in (\ref{6.10}). However, the coefficient of the first order term in (\ref{6.10}) is bounded, theorem \ref{thm9} works for our case. Therefore we have
\begin{equation}\label{6.11}
(\dashint_{B_p(r)}w^{\lambda})^{\frac{1}{\lambda}} \leq C \inf\limits_{B_p{(\frac{r}{16})}}w.
\end{equation}

Define a cut-off function $\varphi$ depending only on the distance to $p$, given by
\begin{equation}\label{6.12}
\varphi(r) = \left\{
\begin{array}{rl} 1 & 0 \leq r \leq 1 \\
2-r & 1< r < 2 \\ 0 & r\geq 2.
\end{array}\right.
\end{equation}

Multiplying (\ref{6.10}) on both side by $\varphi^2 w^{-\frac{1}{3}}$, after the integration,
 we get
$$\int \varphi^2 w^{-\frac{1}{3}}\Delta w + \frac{2n-4}{n-1}\langle \nabla h, \nabla w\rangle w^{-\frac{1}{3}}\varphi^2
 + uw^{-\frac{1}{3}}\varphi^2 \leq 2(n-1)\int w^{\frac{2}{3}}\varphi^2.$$

Integration by parts gives
 \begin{equation}\label{6.13}
\begin{aligned}
2(n-1)\int w^{\frac{2}{3}}\varphi^2 &\geq \int uw^{-\frac{1}{3}}\varphi^2 - \int \langle \nabla(\varphi^2w^{-\frac{1}{3}}), \nabla w\rangle + \frac{3(2n-4)}{n-1}\int \langle \nabla h, \nabla w^{\frac{1}{3}}\rangle w^{\frac{1}{3}}\varphi^2 \\& = \int uw^{-\frac{1}{3}}\varphi^2 - 6\int \varphi w^{\frac{1}{3}}
\langle \nabla \varphi, \nabla w^{\frac{1}{3}}\rangle + 3\int \varphi^2|\nabla w^{\frac{1}{3}}|^2 \\&+
\frac{3(2n-4)}{n-1}\int \langle \nabla h, \nabla w^{\frac{1}{3}}\rangle w^{\frac{1}{3}}\varphi^2.
\end{aligned}
\end{equation}

Using Schwartz inequality, we find
\begin{equation}\label{6.14}
\begin{aligned}
&- 6\int \varphi w^{\frac{1}{3}}
\langle \nabla \varphi, \nabla w^{\frac{1}{3}}\rangle \geq -\delta \int \varphi^2|\nabla w^{\frac{1}{3}}|^2 - \frac{9}{\delta}\int |\nabla \varphi|^2 w^{\frac{2}{3}},\\&
\frac{3(2n-4)}{n-1}\int \langle \nabla h, \nabla w^{\frac{1}{3}}\rangle w^{\frac{1}{3}}\varphi^2 \geq
-\delta \int \varphi^2|\nabla w^{\frac{1}{3}}|^2 - \frac{C_1}{\delta}\int |\nabla h|^2\varphi^2 w^{\frac{2}{3}}.
\end{aligned}
\end{equation}
where $C_1$ is a constant depending only on $n$.

We take $\delta = 1$. Noting that $|\nabla h| \leq n - 1, |\nabla \varphi| \leq 2$, we yield from (\ref{6.13}) and (\ref{6.14}) that
\begin{equation}\label{6.15}
 \begin{aligned}
C_2\int_{B_p(2)} w^{\frac{2}{3}} &\geq \int_{B_p(1)}uw^{-\frac{1}{3}}\\&
\geq (n-1)^{-\frac{2}{3}}\int_{B_p(1)}u.
\end{aligned}
\end{equation}
where $C_2$ is a positive constant depending only on $n$.
Using (\ref{6.9}), (\ref{6.11}), (\ref{6.15}) and the relative volume comparison theorem, we find
\begin{equation}\label{6.16}
\dashint_{B_p(1)} u \leq C_3(w(p))^{\alpha}
\end{equation}
where $C_3, \alpha$ are positive constants depending only on $n$. Following (\ref{6.5}), (\ref{6.16}),
we obtain
\begin{equation}\label{6.17}
\dashint_{B_p(1)}2\sum\limits_{i\neq j}h^2_{ij}+\frac{2n}{n-1}h_{11}^2
+ 2\sum\limits_{i\neq1}(\frac{\Delta h- h_{11}}{n-1}-h_{ii})^2\leq C(n)(w(p))^{\alpha}.
\end{equation}
(\ref{6.2}), (\ref{6.8}), (\ref{6.9}), (\ref{6.11}) imply
\begin{equation}\label{6.18}
\dashint_{B_p(1)}(\Delta h + (n - 1)^2)^2 \leq C(n)(w(p))^{\beta}.
\end{equation}
where $\beta$ is a positive constant depending only on $n$.
(\ref{6.17}) and (\ref{6.18}) imply
\begin{equation}\label{6.19}
\dashint_{B_p(1)}2\sum\limits_{i\neq j}h^2_{ij}+\frac{2n}{n-1}h_{11}^2
+ 2\sum\limits_{i\neq1}(1-n-h_{ii})^2\leq C(n)(w(p))^{\gamma}.
\end{equation}
where $\gamma = \gamma(n) > 0$.
 Now we would like to use the K\"ahler structure of $M$.
 Applying (\ref{1.1}) to $h$, we find
\begin{equation}\label{6.20}
\frac{1}{2}\langle \nabla h, \nabla(\sum\limits_{\gamma \neq 1}h_{\gamma \overline{\gamma}})\rangle = h_{1\overline{1}}\Delta h -|h_{\alpha \overline{\beta}}|^2+ Re (div Y)
\end{equation}
where $Y = \sum\limits_{\gamma \neq 1}h_{\overline{\alpha}}h_{\alpha \overline{\gamma}}e_{\gamma}$.

Suppose at a point $p \in M$,
\begin{equation}\label{6.21}
|\nabla h(p)| > n-1-\epsilon
\end{equation}
 where $\epsilon$ is a very small constant. Then
\begin{equation}\label{6.22}
w(p) \leq C(n)\epsilon.
\end{equation}
Integrating (\ref{6.20}) on the geodesic ball $B_p(r)$, we get
\begin{equation}\label{6.23}
\begin{aligned}
&-\int_{B_p(r)}\Delta h \sum\limits_{\alpha \neq 1}(h_{\alpha\overline{\alpha}}+2m-1) + \frac{1}{2}\int_{\partial B_p( r)} \sum\limits_{\alpha \neq 1}(h_{\alpha\overline{\alpha}}+2m-1)\langle \nabla h, \nabla r\rangle \\&= \int_{B_p(r)} h_{1\overline{1}}\Delta h - |h_{\alpha\overline{\beta}}|^2 + \frac{1}{2}Re\int_{\partial B_p(r)} \sum\limits_{\alpha \neq 1} |\nabla h|h_{1\overline{\alpha}}\langle e_{\alpha}, \nabla r\rangle.
\end{aligned}
\end{equation}

Define the annulus $A = \{x\in M\|\frac{1}{2} \leq d(x, p) \leq 1\}$.
Integrating (\ref{6.23}) with respect to $r$ from $\frac{1}{2}$ to $1$, dividing both side by $Vol(B_p(1))$, we find
\begin{equation}\label{6.24}
\begin{aligned}
&-\int\limits_{\frac{1}{2}}^{1}\frac{\int_{B_p(r)}\Delta h \sum\limits_{\alpha \neq 1}(h_{\alpha\overline{\alpha}}+2m-1)}{Vol(B_p(1))}dr +
\frac{1}{2}\frac{\int_{A} \sum\limits_{\alpha \neq 1}(h_{\alpha\overline{\alpha}}+2m-1)\langle \nabla h, \nabla r\rangle }{Vol(B_p(1))} \\&= \int\limits_{\frac{1}{2}}^{1}\frac{\int_{B(P,r)} h_{1\overline{1}}\Delta h - |h_{\alpha\overline{\beta}}|^2 }{Vol(B_p(1))}dr+ \frac{1}{2}\frac{Re\int_{A} \sum\limits_{\alpha \neq 1} |\nabla h|h_{1\overline{\alpha}}\langle e_{\alpha}, \nabla r\rangle}{Vol(B_p(1))}.
\end{aligned}
\end{equation}

In view of (\ref{6.19}), after the complexification, we obtain
\begin{equation}\label{6.25}
\dashint_{B_p(1)}\sum\limits_{\alpha\neq \beta}|h_{\alpha\overline\beta}|^2+(h_{1\overline 1}+\frac{2m-1}{2})^2
+ \sum\limits_{\alpha\neq1}(1-2m-h_{\alpha\overline\alpha})^2\leq C(n)\epsilon^{\gamma}.
\end{equation}
Following (\ref{6.25}) and the relative volume comparison, we see that up to a negligible error, we can replace the complex hessian of $h$ in (\ref{6.24}) by the corresponding constants in (\ref{6.25}). Explicitly,
\begin{equation}\label{6.26}
h_{\alpha\overline\beta} \rightarrow \left\{
\begin{array}{rl}  0 &  \alpha \neq \beta \\ 1-2m & \alpha=\beta, \alpha \neq 1 \\
\frac{1-2m}{2} & \alpha = \beta = 1 \\
\end{array} \right.
\end{equation}

In order to get a contradiction to (\ref{6.21}), it suffices to find a gap between the LHS and the RHS of (\ref{6.24}) if we replace $h_{\alpha\overline\beta}$ by (\ref{6.26}).
Plugging (\ref{6.26}) in (\ref{6.24}), the LHS is 0, the RHS is
\begin{equation}\label{6.27}
\begin{aligned}
&(\frac{1-2m}{2}(\frac{1-2m}{2}+(m-1)(1-2m)) - (\frac{1-2m}{2})^2\\&-(m-1)(1-2m)^2)\int\limits_{\frac{1}{2}}^{1}\frac{Vol(B(P, r))}{Vol(B(P, 1))}dr\\&=-\frac{(2m-1)^2(m-1)}{2}\int\limits_{\frac{1}{2}}^{1}\frac{Vol(B(P, r))}{Vol(B(P, 1))}dr \\&\leq -\frac{(2m-1)^2(m-1)}{2}C(n)
\end{aligned}
\end{equation}
where $C(n)$ is a positive constant depending only on $n$.

The proof of theorem \ref{thm4} is complete.
\qed


\begin{thebibliography}{99}

\bibitem{[BC]}  Bishop. R.L, Crittenden. R.J.:\emph{ Geometry of Manifolds}, Pure and Applied Math., Vol.
XV. New York-London: Academic Press 1964.

\bibitem{[CC]}  J. Cheeger, T. Colding, \emph{ Lower bounds on Ricci curvature and the almost rigidity of warped products}. Ann. of Math. (2) 144 (1996), no. 1, 189-237.

\bibitem{[CE]}  J. Cheeger, D. Ebin.: \emph{Comparison Theorems in Riemannian Geometry}, 2000 Mathematics Subject Classification.Primary 53C20;
Secondary 58E10.

\bibitem{[C]}   S. Y. Cheng, \emph{ Eigenvalue comparison theorems and geometric applications}, Math. Z. 143 (1975), 289-297.

\bibitem{[CY]}   S. Y. Cheng and S. T. Yau, \emph{ Differential equations on Riemannian manifolds and their geometric applications}, Comm. Pure Appl. Math. 28 (1975), 333-354.

\bibitem{[LI]}  P. Li, \emph{ Lecture notes on geometric analysis}, Lecture Notes Series, 6, Research Institute of Mathematics and Global Analysis Research Center, Seoul National University, Seoul, 1993.

\bibitem{[LS]}   P. Li and R. Schoen, \emph{ $L^p$ and mean value properties of subharmonic functions on Riemannian manifolds}, Acta Math. 153 (1984), 279-301.

\bibitem{[LW]}  P. Li and J. Wang, \emph{ Comparison theorem for K\"{a}hler manifolds and positivity of spectrum}, J.
Diff. Geom. 69 (2005), 43-74.

\bibitem{[LW1]}   P. Li and J. Wang, \emph{ Complete manifolds with positive spectrum, II}, J. Diff. Geom, 62(2002), 143-162.

\bibitem{[L]}   G. Liu,  \emph{ Local volume comparison for K\"ahler manifolds}, to appear in Pacific. J. Math.

\bibitem{[Y]}  S. T. Yau, \emph{ Harmonic functions on complete Riemannian manifolds}, Comm. Pure Appl. Math. 28(1975), 201-228.
\end{thebibliography}
  \end{document}